%% file: Main.tex
\documentclass[11pt, twoside, leqno]{amsart}  

\input{preamble}  
\usepackage{amsaddr} 
\pgfplotsset{compat=1.17}
\makeatletter
\@namedef{subjclassname@2020}{\textup{2020} Mathematics Subject Classification}
\makeatother

\title[Asymptotic-preserving finite element analysis of Westervelt equations]{Asymptotic-preserving finite element analysis of Westervelt-type wave equations}


\author[V.\ Nikoli\'c]{Vanja Nikoli\'c}

\address{ 
	Department of Mathematics, 
	Radboud University   \\ 
	Heyendaalseweg 135,
	6525 AJ Nijmegen, The Netherlands}
\email{vanja.nikolic@ru.nl} 

\begin{document}
\vspace*{8mm}
\begin{abstract}
\input{abstract.tex}
\end{abstract}
\vspace*{-7mm}
\maketitle           
\input{content.tex}
\input{appendix.tex}
\bibliography{references}{}
\bibliographystyle{siam} 
\end{document}

%% file: preamble.tex
\usepackage{lipsum}
\usepackage{amsfonts}
\usepackage{graphicx}
\usepackage{epstopdf}
\usepackage{algorithmic}
\usepackage{calligra}
\usepackage{amsfonts,amsmath,amsthm,amssymb}
\usepackage{mathtools}
\usepackage{hyperref}
\usepackage{autonum}
\usepackage{hhline}
\usepackage{array}
\usepackage{diagbox}
\usepackage{tcolorbox}
\usepackage{mdframed}
\usepackage{multicol}
\usepackage{graphicx}
\usepackage{subcaption}
\usepackage{moreverb}
\usepackage{bbm}
\usepackage[margin=1.34in]{geometry}
\usepackage{todonotes}
\usepackage{scalerel,amssymb}
\allowdisplaybreaks
\usepackage{mathrsfs}  
\usepackage{lineno}
\usepackage{todonotes}
\usepackage{tikz}
\usepackage{pgfplots}
\usetikzlibrary{arrows.meta}
\usepackage{siunitx}
\usepackage[numbers,sort&compress]{natbib}
\definecolor{mygreen}{HTML}{43a047}
\usepackage{subcaption}
\usepackage{doi}
\usepackage{enumitem}
\usepackage{tikz-cd}
\usepackage{adjustbox}
\usepackage{appendix}
\usepackage{wrapfig}
\newcommand\usolspace{\mathcal{U}^r}
\newcommand\Lconv{\ast}
\newcommand{\bfmu}{\boldsymbol{\mu}}
\newcommand{\bfxi}{\boldsymbol{\xi}}

\def \frakK{\mathfrak{K}}
\def\calM{\mathcal{M}}

\def\ulal{\underline{\mathfrak{a}}}
\def\olal{\overline{\mathfrak{a}}}

\newcommand{\Om}{\Omega}
\newcommand{\eps}{\varepsilon}

\DeclareMathOperator*{\esssup}{ess\,sup}
\def\tCinv{\tilde{C}_{\textup{inv}}}
\newcommand{\CHonethree}{C_{H^1 \hookrightarrow L^3}}
\newcommand{\ut}{u_t}
\newcommand{\utt}{u_{tt}}
\newcommand{\uttt}{u_{ttt}}

\newcommand{\ds}{\, \textup{d} s }

\newcommand{\dxs}{\, \textup{d}x\textup{d}s}

\newcommand{\inttO}{\int_0^t \int_{\Omega}}
\newcommand{\intt}{\int_0^t}

\newcommand{\intO}{\int_{\Omega}}

\newcommand{\R}{\mathbb{R}} 
 
\newcommand{\Hone}{H^1(\Omega)}

\def\Ltwo{L^2(\Omega)}
\def\Lthree{L^3(\Omega)}

\def\Lsix{L^6(\Omega)}
   \def\Linf{L^\infty(\Omega)}

  \def \Clin{C_{\textup{lin}}}

\newtheorem{theorem}{Theorem}

\newtheorem{proposition}{Proposition}

\newtheorem{corollary}{Corollary}

\numberwithin{lemma}{section}
\numberwithin{proposition}{section}
\numberwithin{theorem}{section}
\numberwithin{equation}{section}
\makeatletter
\newcommand{\leqnomode}{\tagsleft@true}
\newcommand{\reqnomode}{\tagsleft@false}
\makeatother

\newcommand{\Hrzero}{{H^r(\Omega)\cap H_0^1(\Omega)}}
\newcommand{\Hr}{{H^r(\Omega)}}
\definecolor{grey}{rgb}{0.5,0.5,0.5}
 
\newcommand{\ueps}{u^{\varepsilon}}
\newcommand{\uteps}{u_t^{\varepsilon}}

\newcommand{\calF}{\mathcal{F}}
\newcommand{\Bh}{B_h}

\def\aaa{\mathfrak{a}}

\def\aaalin{\mathfrak{a}}
\def\aaalint{\mathfrak{a}_t}
\def\aaalintt{\mathfrak{a}_{tt}}

\def\psih{u_h}
\def\psiht{u_{ht}}
\def\alphah{\alpha_h}
\def\alphaht{\alpha_{ht}}
\def\alphahtt{\alpha_{htt}}
\def\psihtt{u_{htt}}

\def\uh{u_h}

\def\uht{u_{ht}}
\def\uhtt{u_{htt}}

\def\eh{e_h}
\def\eht{e_{ht}}

\def \Rh{\mathcal{R}_h}
\def \Ih{I_h}

\def \Cinv{C_{\textup{inv}}}

\def \Capproxint{C_{\textup{app}}^{I_h}}
\def \Capproxritz{C_{\textup{app}}^{\mathcal{R}_h}}
\def \Capproxritzz{C_{\textup{app}}^{\mathcal{R}_h}}
\def \Cstint{C_{\textup{st}}^{I_h}}

\definecolor{darkgreen}{rgb}{0,0.5,0}

\theoremstyle{definition}
\newtheorem{exmp}{Example}[section]

\usepackage{marginnote}

\newcommand{\revisedd}[2]{#2}
\newcommand{\reviseV}[1]{\color{black} #1}

%% file: abstract.tex
Motivated by numerical modeling of ultrasound waves, we investigate conforming finite element approximations of quasilinear and possibly nonlocal equations of Westervelt type. These wave equations involve either a strong dissipation or damping of fractional-derivative type and we unify them into one class by introducing a memory kernel that satisfies non-restrictive regularity and positivity assumptions. As the involved damping parameter is relatively small and can become negligible in certain (inviscid) media, it is important to develop methods that remain stable as the said parameter vanishes. To this end, the contributions of this work are twofold. First, we determine sufficient conditions under which conforming finite element discretizations of (non)local Westervelt equations can be made robust with respect to the dissipation parameter. Secondly, we establish the rate of convergence of the semi-discrete solutions in the singular vanishing dissipation limit. The analysis hinges upon devising appropriate energy functionals for the semi-discrete solutions that remain uniformly bounded with respect to the damping parameter. Numerical experiments are provided to illustrate the theory.   

%% file: content.tex
\section{Introduction}
At high frequencies or amplitudes, propagation of sound waves is described by quasilinear wave equations. In complex tissue-like media, additionally, nonlocal attenuation effects become apparent. In practice, the involved dissipation parameter is relatively small and becomes negligible in inviscid media. It is thus important to establish conditions under which numerical simulation methods remain stable and accurate across different attenuation scales. Simulation of such nonlinear and nonlocal models is especially relevant for advancing medical applications of ultrasound waves in imaging~\cite{szabo2004diagnostic} and treatment of solid tumors~\cite{kennedy2005high}. With this motivation in mind, we investigate finite element discretizations  of the following class of nonlinear wave equations:
\begin{equation} \label{model_eq}
	\begin{aligned}
	(\mathfrak{a}(u^\eps) u^\varepsilon_t)_t-c^2\Delta u^\varepsilon-\varepsilon  \frakK*\Delta \ut^\varepsilon=f,
	\end{aligned}
\end{equation}
that are robust with respect to the parameter $\eps\, \revisedd{R2}{\ll}\, 1$ , where
\[
\aaa(\ueps)=1+k\ueps, \quad k \in \R,
\]
and $*$ denotes the Laplace convolution in time. Originally derived by Westervelt in~\cite{westervelt1963parametric} for inviscid media with $\eps=0$, in thermoviscous media it involves $\frakK=\delta_0$ (the Dirac delta distribution) and thus the strong damping $-\eps \Delta \uteps$. The damping parameter is then known as the sound diffusivity~\cite{lighthill1956viscosity} as it contributes to the parabolic-like character of the model. The vanishing limit is thus singular as there is a change from parabolic-like to hyperbolic evolution.  In complex media, $\frakK$ is weakly singular. It may be given by the Abel fractional kernel:
\begin{equation} \label{Abel_krenel}
	\frakK= \frac{1}{\Gamma(1-\alpha)}t^{-\alpha} \quad \alpha \in (0,1),
\end{equation}
where $\Gamma(\cdot)$ is the Gamma function. In that case the resulting dissipation term is $-\eps \Delta \textup{D}_t^\alpha u$, where $\textup{D}_t^\alpha$ is the Caputo--Djrbashian derivative of order $\alpha$. Mittag-Leffler-type kernels are also of interest; we refer to~\cite{kaltenbacher2022limiting} and Section~\ref{Sec:TheoreticalPreliminaries} for details. We unify these different cases of interest into one model in \eqref{model_eq} by imposing general non-restrictive regularity and sign conditions on $\frakK$; see Section~\ref{Sec:TheoreticalPreliminaries}.  \\
\indent \indent The goal of this work is twofold. First, we wish to determine sufficient conditions for a uniform conforming finite element discretization of \eqref{model_eq} with respect to $\varepsilon$. In other words, we wish to arrive at stability and \emph{a priori} error estimates for the corresponding semi-discrete solution where the involved constant does not blow up as $\varepsilon \rightarrow 0$.  Secondly, we aim to establish whether the vanishing dissipation asymptotics is preserved on the discrete level. These properties are visualized in the diagram in Figure~\ref{Fig:Diagram}. The present work fits into the wider class of studies on asymptotic-preserving numerical methods; see, e.g.,~\cite{jin2022asymptotic, degond2017asymptotic, feireisl2018asymptotic, filbet2013analysis}, and the references provided therein. Remarkably, such methods permit discretizing the preturbed and limit problems with the same discretization parameters. \revisedd{R1}{We also note that the limiting small-parameter behavior of mathematical models and their discretizations is of further interest for studying, for example, homogenization methods. This field is quite rich; we point the readers to, e.g., the results of~\cite{jager2023approximation, d2024homogenization} and, in the context of linear wave equations in heterogeneous media, to~\cite{abdulle2020effective, owhadi2008numerical}, and the references contained therein.}
\begin{figure}[h]
	\adjustbox{scale=1.1,center}{
		\begin{tikzcd}[row sep=huge, column sep = huge]
		\ref{Ph_nl} \arrow[r, dashed, "h \rightarrow 0"] \arrow[d, dashed, "\eps \rightarrow 0"]
		& \ref{P_nl}  \arrow[d, "\eps \rightarrow 0"] \\
		\ref{Ph_nl_zero} \arrow[r                                                                                                                                                                                                                                                                                                                                                                                                                                                                                                                                                                                                                                          , "h \rightarrow 0"]
		&  \ref{P_nl_zero}
		\end{tikzcd}
	}
	\caption{Asymptotic Preserving diagram: This work establishes the connection between \ref{Ph_nl} and \ref{P_nl} as $h \rightarrow$ 0, uniformly in $\varepsilon$, as well as between \ref{Ph_nl} and \ref{Ph_nl_zero} as $\varepsilon \rightarrow 0$.} \label{Fig:Diagram}
\end{figure} 

\subsection*{Main contributions} Our main results provide sufficient conditions under which the following convergence rates hold:
		\begin{equation} \label{eps_conv_1}
\begin{aligned}
\|u^\eps-\uh^\eps\|_{W^{2,\infty}(0,T; \Ltwo)}+h\|\nabla(u^\eps-\uh^\eps)\|_{W^{1,\infty}(0,T; \Ltwo)} \leq C h^r\revisedd{R2}{,}
\end{aligned}
\end{equation}
with $d/2<r \leq p+1$ and $p$ being the polynomial degree of the finite element space, and
		\begin{equation} \label{eps_conv_2}
	\begin{aligned}
	\begin{multlined}[t]	\|\uh^\eps-\uh^0\|_{L^\infty(0,T; \Ltwo)}+	\|\uht^\eps-\uht^0\|_{L^\infty(\Ltwo)}+	\|\nabla(\uh^\eps-\uh^0)\|_{L^\infty(\Ltwo)} \\ \leq C {\eps}, \end{multlined}
	\end{aligned}
\end{equation}
where the involved constants do not depend on $h$ or $\eps$.  Estimate \eqref{eps_conv_1}  allows keeping the discretization parameter $h$ fixed across different scales of $\eps$. Estimate \eqref{eps_conv_2} establishes the error made when exchanging the damped and undamped approximations. The nonlocal term in \eqref{model_eq} can also be understood as a numerical regularization of the undamped problem. Such artificial diffusion (in particular, with $\frakK=\delta_0$) is commonly added to reduce the effects of numerical dispersion in the simulation of quasilinear acoustic waves; see, for example,~\cite{schenke2022explicit}. Thus error bound \eqref{eps_conv_2} also allows determining the error of this numerical relaxation.  
\subsection*{Related works and the present approach} To the best of our knowledge, this is the first asymptotic study of nonlinear semi-discrete acoustic equations. In terms of related works, we point out the non-uniform finite element analysis of the local Westervelt equation (with $\frakK=\delta_0$ and $\eps>0$). In this strongly damped setting, the Dirichlet boundary value problem for the Westervelt equation is known to be globally well-posed with an exponential decay of the energy of the system; see~\cite{kaltenbacher2009global, meyer2011optimal}. \\
\indent In the undamped case ($\eps=0$), smooth solutions are expected to exist only locally in time; see~\cite{dorfler2016local} for a well-posedness analysis. Finite element analysis of the inviscid Westervelt equation  follows as a special case of the results in~\cite{hochbruck2022error, dorich1173strong}. Uniform-in-$\eps$ error estimates for a mixed finite element approximation of the local Westervelt equation in the so-called potential form (that is, with $\frakK=\delta_0$ and $\aaa=1+k \ut^\eps$) is available in \cite{Meliani2022}. Analysis of a time-discretization of \eqref{model_eq} with fractional-type damping has been performed in~\cite{baker2022numerical}. \\
\indent The core of the analysis arguments of our approach lies in devising semi-discrete energy functionals for which the approximate solution $\uh^\eps$ is uniformly bounded in $\eps$. The uniform energy analysis is first carried out on a linearized problem with $\aaa=\aaa(x,t) \geq \ulal >0$ and then transferred to a nonlinear one through a fixed-point argument in the spirit of, e.g.,~\cite{makridakis1993finite, ortner2007discontinuous, nikolic2019priori}.  For this fixed-point argument to work, it is essential that a uniform bound is established  on $\uhtt^\eps$ in a suitable norm. Compared to the non-uniform finite-element analysis, the present approach leads to having to consider also the time-differentiated problem and involve higher-order energy functionals; cf.~\cite{nikolic2019priori}. Note that the energy maneuvers are considerably restricted by ({\bf 1})  ensuring the non-degeneracy of the leading factor $1+k \uh^\eps \geq \ulal>0$, and ({\bf 2}) the positivity properties one can expect from the (in general, fractional-type)  kernels. Thus there is a delicate interplay of nonlinearity, nonlocality, and potential degeneracy of the problem to consider.
\subsection*{Organization of the paper} We organize the rest of the exposition as follows. In Section~\ref{Sec:TheoreticalPreliminaries}, we discuss the assumptions on the kernel and provide details on the semi-discretization. Section~\ref{Sec:LinearAnalysis} contains the uniform stability and \emph{a priori} error analysis of a linearized problem with a variable leading coefficient. Section~\ref{Sec:NonlinearAnalysis} relates these results to the nonlinear problem using Schauder's fixed-point theorem, under suitable smallness conditions on the discretization parameter and the exact solution, independently of $\eps$. In Section~\ref{Sec:EpsLimit}, we prove the error bound \eqref{eps_conv_2}. The key theoretical results are contained in Theorems~\ref{Thm:WellpNl} and~\ref{Thm:epsConv_lower}. Finally, in Section~\ref{Sec:Numerics}, we illustrate the theory with numerical examples.
\section{Problem setting} \label{Sec:TheoreticalPreliminaries}
Let $\Omega \subset \mathbb{R}^d$ be \revisedd{R2}{an interval ($d=1$), or} a convex polygonal ($d=2$) or polyhedral ($d=3$) domain. Given final time $T>0$, we consider the following initial boundary-value problem for the (in general) nonlocal Westervelt equation: 
\begin{equation} \label{P_nl} \tag{\ensuremath{{\bf {P}}^{\varepsilon}}}
\left\{ \begin{aligned}
&(((1+k u^\eps)\ut^\eps)_t, \phi)_{L^2}+(c^2\nabla u^\eps+\varepsilon  \frakK*\nabla \ut^\eps, \nabla \phi)_{L^2}= (f, \phi)_{L^2}, \\[2mm]
&\text{for all } \phi\in H_0^1(\Omega) \text{ a.e.\ in time, with}  \\[2mm]
&(u^\eps, \ut^\eps)\vert_{t=0}=	(u_{0}, u_{1}),
\end{aligned} \right.
\end{equation}
where $(\cdot, \cdot)_{L^2}$ denotes the scalar product on $L^2(\Omega)$. The limiting problem is given by
\begin{equation} \label{P_nl_zero} \tag{\ensuremath{{\bf {P}}^{0}}}
\left\{ \begin{aligned}
&(((1+k u^0)\ut^0)_t, \phi)_{L^2}+(c^2\nabla u^0, \nabla \phi)_{L^2}= (f, \phi)_{L^2}, \\[2mm]
&\text{for all } \phi\in H_0^1(\Omega) \text{ a.e.\ in time with}  \\[2mm]
&(u^0, \ut^0)\vert_{t=0}=	(u_{0}, u_{1}).
\end{aligned} \right.
\end{equation}
We next impose relatively mild conditions on the kernel $\frakK$ that allow us to cover the relevant examples from the literature.
\subsection{Assumptions on the memory kernel} We make the following regularity and positivity assumptions on the memory kernel: \vspace*{1mm}
\begin{enumerate}[label=(\textbf{$\bf \mathcal{A}_\arabic*$})]
	\item \label{Aone}	 	 $\frakK \in \{\delta_0\} \cup L^1(0,T)$;  \\
	\item \label{Atwo}  For all $y\in L^2(0, T; \Ltwo)$, 
	\begin{equation}
		\int_0^{t} \intO \left(\frakK* y \right)(s) \,y(s)\dxs \geq 0, \quad t \in (0,T].
	\end{equation} 
\end{enumerate}
The case $\frakK=\delta_0$ satisfies both assumptions and leads to the strongly damped Westervelt equation. We expect the analysis below to also generalize to measures $\frakK *\in \mathcal{M}(0,t)$ that satisfy \ref{Atwo}, however with somewhat increased technicality. As a reminder of this possibility, we use the following norm below: 
\[
\|\frakK\|_{\mathcal{M}(0,T)}=\begin{cases}
\	1 &\text{ if }\frakK=\delta_0,\\[1mm]
\	\|\frakK\|_{L^1(0,T)}&\text{ if }\frakK\in L^1(0,T).
\end{cases}
\]
\indent Assumptions \ref{Aone} and \ref{Atwo} are also satisfied by fractional-derivative-type kernels. In particular, they hold for the Abel kernel in \eqref{Abel_krenel} and thus the below semi-discrete theory applies to the Westervelt equation with time-fractional damping:
\begin{equation} \label{West_frac}
	\begin{aligned}
		(\aaa(u^\eps) \ut^\eps)_t-c^2\Delta u^\eps-\varepsilon \Delta \textup{D}_t^{\alpha} u^\eps=f, \qquad \alpha \in (0,1);
	\end{aligned}
\end{equation}
see~\cite{kaltenbacher2022limiting}. Given $w \in W^{1,1}(0,T)$, $\textup{D}_t^{\alpha}$ denotes the Djrbashian--Caputo fractional derivative of order $\alpha$:
\[
\textup{D}_t^{\alpha}w(t)=\frac{1}{\Gamma(1-\alpha)}\int_0^t (t-s)^{-\alpha} w_t(s) \, \textup{d} s,
\]
where $\Gamma(\cdot)$ is the Gamma function; see,~\cite[Ch.\ 1]{kubica2020time} and~\cite[Ch.\ 2.4.1]{podlubny1998fractional}. \\
\indent Kernels of Mittag-Leffler type given by
\begin{equation} \label{MT_kernel}
	\frakK (t)= t^{\beta-1} E_{\alpha, \beta}(-t^{\alpha}), \quad  \alpha \in (0, 1]\revisedd{R2}{,}
 \end{equation}
with $\beta \in \{1, \alpha, 2 \alpha-1\}$ arise in nonlinear acoustics when the heat flux laws of Compte--Metzler type written in the Gurtin--Pipkin form are used within the system of governing equations of sound motion; see~\cite{kaltenbacher2022limiting} for the modeling details. Above $E_{\alpha, \beta}$ denotes the generalized Mittag-Leffler function~\cite[Ch.\ 2]{kubica2020time}:
\begin{equation}
	E_{\alpha, \beta} (t)= \sum_{k=0}^\infty \frac{t^k}{\Gamma(\alpha k+\beta)}.
\end{equation}
Such kernels also satisfy assumptions \ref{Aone} and \ref{Atwo}; see, e.g.,~\cite{kaltenbacher2022limiting, gorenflo2020mittag}. \\
\indent In general, assumption \ref{Atwo} can be verified using the Fourier analysis techniques from~\cite[Lemma 2.3]{Eggermont1987} or, in case of completely monotone non-constant kernels, by relying on~\cite[Lemma 5.2]{kaltenbacher2022limiting}.
\subsection{Assumptions on the exact solution} 
\indent In what follows, we will assume that \eqref{P_nl} has a solution, such that
\begin{equation} \label{def_solspace}
	u^\eps \in \usolspace=W^{3,1}(0,T; \Hrzero),
\end{equation}
where $r> d/2$. We furthermore assume that the solution is uniformly bounded in $\eps$ in the $\|\cdot \|_{\usolspace}$ norm. \\
\indent Results of this type have been recently obtained in the literature, under somewhat stronger coercivity assumptions on the memory kernel; see~\cite{kaltenbacher2022limiting, kaltenbacher2023limiting, kaltenbacher2022parabolic} for details.  $\eps$-uniform well-posedness of \eqref{P_nl} with
\[
\begin{aligned}
 &u^\eps \in L^\infty(0,T;  H^3(\Omega) \cap H_0^1(\Omega)) \cap W^{1, \infty}(0,T; H^2(\Omega) \cap H_0^1(\Omega)) 
\cap H^2(0,T; H_0^1(\Omega)), \\
& \Delta u^\eps \vert_{\partial \Omega} =0,
\end{aligned}
\]
has been proven in~\cite[Theorem 3.1]{kaltenbacher2022inverse} assuming (in addition to \ref{Aone}) that there exists $C>0$, independent of $\varepsilon$, such that
\begin{equation}
	\int_0^{t} \intO \left( \frakK* y \right)(s) \,y(s)\dxs \geq 	C \intt \|( \frakK* y)(s)\|^2_{\Ltwo} \ds.
\end{equation}
This assumption is fulfilled as well by the kernels discussed above (with $\alpha>1/2$ if $\beta=2 \alpha-1$ in \eqref{MT_kernel}); see~\cite[Section 5]{kaltenbacher2022limiting} for the verification. \\
\indent The regularity we assume in \eqref{def_solspace} (and, in particular, $W^{3,1}$ regularity with respect to time), can be proven for sufficiently smooth data by modifying the arguments of \cite[Theorem 3.1]{kaltenbacher2023limiting}, which covers the nonlinearities of the form $a(u^\eps)=1+ k u^\eps_t$, under an additional assumption on the kernel:
\[
	\int_0^t\bigl((\frakK * y_t)(s), y(s)\bigr)_{L^2(\Omega)}\ds\geq - C\, \|y(0)\|^2_{L^2(\Omega)}, \quad
y\in W^{1,1}(0,t;L^2(\Omega)),
\]
which is again satisfied by all kernels mentioned above; see~\cite[Section 3]{kaltenbacher2023limiting}. In agreement with the assumed regularity in time,  we also mention that with fractional-type kernels the solutions of Westervelt's equation with $f=0$ are expected to obey the time asymptotics:
\[
u_{tt}(t) \sim w_0+ z_0 t^\mu+ \revisedd{R2}{\mathcal{O}}(t) \ \text{as} \ t \rightarrow 0\revisedd{R2}{,}
\]
with $\mu \in (0,1)$; see~\cite[Section 5]{baker2022numerical} for details.
\subsection{Notation} Below we continue to use $(\cdot, \cdot)_{L^2}$ to denote the scalar product on $\Ltwo$.  When writing norms on Bochner spaces, we omit the temporal domain when it is $(0,T)$; for example, $\|\cdot\|_{L^p(L^q(\Om))}$ denotes the norm on $L^p(0,T; L^q(\Omega))$. We occasionally use $x \lesssim y$ to denote $x \leq C y$, where $C>0$ is a constant that does not depend on the discretization parameter or $\eps$. Likewise, $C_{1}$, $C_2$ below denote generic positive constants that do not depend on the discretization parameter or $\eps$. We use $x \lesssim_T y$ when the hidden constant depends on $T$ in such a way that it tends to $\infty$ as $T \rightarrow \infty$; this is often the case after employing a Sobolev embedding in time.

\subsection{Semi-discretization} We consider the discretization in space by continuous piecewise polynomial finite elements that vanish on the boundary.  For $h \in (0, \overline{h}]$, let $\mathcal{T}_h$  be a discretization of $\Omega$ made of intervals $K$ (in $\mathbb{R}$), triangles $K$ (in $\mathbb{R}^2$) or of tetrahedrons $K$ (in $\mathbb{R}^3$)  so that $\overline{\Omega}=\cup_{K \in \mathcal{T}_h} \overline{K}$. We denote by $P_{p}(K)$ the space of polynomials on $K$ of degree $p \geq 1$. We then introduce the finite element space as
\begin{align} \label{FEM_space}
V_h=\{u_h \in H_0^1(\Omega): \ u_{h} \vert_{K}\, \in P_p(K), \, \forall K \in \mathcal{T}_h \}.
\end{align}
We assume that $\{\mathcal{T}_h\}_{0<h \leq \overline{h}}$ is a shape-regular and quasiuniform family, which will allow us to exploit the following inverse inequalities for finite element functions:
\begin{subequations}
\begin{align} 
		\| \phi_h\|_{\Linf} \leq&\, \Cinv  h^{-d/2}\|\phi_h\|_{\Ltwo}, \ \phi_h \in V_h, \label{inverse_ineq}\\
			\|\nabla \phi_h\|_{\Ltwo} \leq&\, \tCinv  h^{-1}\|\phi_h\|_{\Ltwo}, \ \phi_h \in V_h; \label{inverse_ineq2}
\end{align}
\end{subequations}
see~\cite[Theorem 4.5.11]{brenner2008mathematical}. \\
\indent Let $\phi \in H^{r}(\Omega)$ with $ d/2< r \leq p+1$. In the upcoming proofs, we will also use the following well-known stability and approximation properties of the standard interpolation operator $\Ih$:
\begin{equation}
	\begin{aligned}
			\|\phi -I_h \phi\|_{\Ltwo} \leq&\, \Capproxint h^{r} \|\phi\|_{\Hr}, \\ 
	\|I_h \phi\|_{L^q(\Omega)} \leq&\, \Cstint \|\phi\|_{L^q(\Omega)}, \quad 1 \leq q \leq \infty;
	\end{aligned}
\end{equation}
see~\cite[Ch.\  4]{brenner2008mathematical}.  In the analysis below, the approximate initial data will be chosen as the Ritz projection of the exact ones and we will employ the Ritz projection of $u^\eps(\cdot, t)$ to decompose the approximation error; these choices are crucial. The Ritz projection of $\phi \in \Hr \cap H_0^1(\Om)$ is defined by
\begin{equation}
	(\nabla (\phi- \Rh \phi), \nabla \phi_h)_{L^2} = 0, \qquad  \forall \phi_h \in V_h,
\end{equation}
and we have
\begin{equation}
	\|\phi-\Rh \phi\|_{\Ltwo}+	h\|\phi-\Rh \phi\|_{\Hone} \leq \Capproxritz h^{r} \|\phi\|_{\Hr};
\end{equation}
see~\cite[Lemma 1.1]{thomee2007galerkin}. 
\section{Uniform finite element analysis of nonlocal linear wave equations with a variable leading coefficient} \label{Sec:LinearAnalysis}
In this section, we analyze a linearized semi-discrete wave equation with a variable principal coefficient. This uniform analysis will serve as a basis for the later fixed-point argument. For ease of notation, in this and the forthcoming section we drop the superscript $\eps$; that is,  \[u=u^\eps \quad \text{and} \quad \uh=\uh^\eps \] in Sections~\ref{Sec:LinearAnalysis} and~\ref{Sec:NonlinearAnalysis}.  We consider the following linearization of the inviscid semi-discrete problem: 
\begin{equation} \label{BK_lin}
\begin{aligned}
(((1+k \alphah)\uht)_t-c^2\Delta_h \uh-\eps  \frakK*\Delta_h\uht , \phi_h)_{L^2}= ( f, \phi_h)_{L^2}\revisedd{R2}{,}
\end{aligned}
\end{equation}
where $\Delta_h:V_h \rightarrow V_h^*$ is the discrete Laplacian operator:
\[
-(\Delta_h v_h, \phi_h)_{L^2}: =(\nabla v_h, \nabla \phi_h)_{L^2}, \qquad  \phi_h \in V_h.
\]
In other words, with the short-hand notation 
\begin{equation}
	\begin{aligned}
		\aaalin=&\,\aaalin(\alphah)=\,1+k \alphah(x,t),
	\end{aligned}
\end{equation}
the linear semi-discrete problem is given by
\begin{equation} \label{Ph} \tag{\ensuremath{{\bf {P}}_{h}^{\varepsilon,\textup{lin}}}}
\left\{ \begin{aligned}
&((\aaa \uht)_t, \phi_h)_{L^2}+(c^2\nabla \psih+\varepsilon  \frakK*\nabla \uht, \nabla \phi_h)_{L^2}= (f, \phi_h)_{L^2}, \\[2mm]
& \text{ for all }\, \phi_h \in V_h \ \text{a.e.\ in time, with} \\[2mm]
&(\uh, \uht)\vert_{t=0}=	(u_{0h}, u_{1h}) =(\Rh u_0, \Rh u_1).
\end{aligned} \right.
\end{equation} 
The proper choice of the approximate initial data is essential for the upcoming stability and error analysis as it will allow us to uniformly bound the initial energy; we refer to Proposition~\ref{Prop:InitialStability} for details. If $\frakK \in L^1(0,T)$, the assumptions on the second approximate initial condition can be relaxed to choosing any $u_{1h} \in V_h$ that satisfies
	\[
	\|u_1-u_{1h}\|_{\Ltwo} +h\|\nabla(u_1-u_{1h})\|_{\Ltwo} \lesssim h^r,
	\]
however, for ease of presentation, we set $u_{1h}=\Rh u_1$.
\subsection*{Strategy in the semi-discrete analysis} In what follows we intend to use the following testing strategy:
\begin{equation} \label{testing_strategy}
\eqref{Ph} \cdot \uht + \eqref{Ph} _t \cdot \uhtt. 
\end{equation}
That is, we will test the linearized problem \eqref{Ph}  with $\phi_h=\uht(t)$ and the time-differentiated linearized problem with $\phi_h=\uhtt(t)$ and integrate over time. The first step will give a bound on a lower-order energy.  The second step will additionally give a bound on a higher-order (in time) energy, which is crucial for the subsequent fixed-point argument, where we will need to have control of the second time derivative of the approximate acoustic pressure. \\
\indent In the process, we will rely on the differentiation formula
\begin{equation} \label{derivative_convolution}
	(\frakK* \uht)_t= \frakK* \uhtt + \eta \frakK(t) \uht(0),
\end{equation}
where we have introduced 
\begin{equation} \label{def_eta}
	\begin{aligned}
	\eta = \begin{cases}
		0 \ &\text{if} \ \, \frakK =\delta_0, \\
		1 \ &\text{otherwise}.
	\end{cases}
\end{aligned}	
\end{equation}
Thus, the time-differentiated equation will have $f_t- \eps \eta \frakK \uht(0) $ as the right-hand side; see~\eqref{BK_lin_diff}.   \\
\indent The two analysis steps within \eqref{testing_strategy} will invoke the following conditions:
\begin{equation}
	\begin{aligned}
	&	\inttO \frakK* \nabla \uht \cdot \nabla \uht \dxs \geq 0, \\
		&	\inttO \frakK* \nabla \uhtt \cdot \nabla \uhtt \dxs \geq 0, 
		\end{aligned}
\end{equation}
which hold thanks to assumption \ref{Atwo}. 

\subsection{Uniform stability of the linearized problem} To facilitate the analysis, we introduce a lower-order energy functional
\begin{equation}
	\begin{aligned}
		E_0[\psih](t)= \|\psiht(t)\|^2_{\Ltwo}+\|\nabla \psih(t)\|^2_{\Ltwo}
	\end{aligned}
\end{equation}
and a higher-order energy functional in time
\begin{equation}
	\begin{aligned}
		E_1[\psih](t)= 	E_0[\psiht](t)=\|\psihtt(t)\|^2_{\Ltwo}+\|\nabla \psiht(t)\|^2_{\Ltwo}
	\end{aligned}
\end{equation}
for $t \in [0,T]$.  Observe that $E_1\revisedd{R2}{[u_h]}(0)$ contains $\|\uhtt(0)\|_{\Ltwo}$, which should stay uniformly bounded in $h$ and $\eps$. We will come back to resolve this issue in  Proposition~\ref{Prop:InitialStability}.  \\
\indent We also point out that we have the equivalence of the norms induced by the energies with the following norms:
\begin{equation}
\begin{aligned}
&\esssup_{t \in (0,T)}	E_0[\psih](t) \sim \|\psih\|^2_{W^{1,\infty}(\Ltwo)}+\|\psih\|^2_{L^{\infty}(\Hone)},\\ 
&\esssup_{t \in (0,T)}	\left(E_0[\psih](t)+E_1[\psih](t)\right) \sim \|\psih\|^2_{W^{2,\infty}(\Ltwo)}+\|\psih\|^2_{W^{1,\infty}(\Hone)}. 
\end{aligned}
\end{equation}
\indent We proceed to show that the linearized problem is stable with respect to the $E_0+E_1$ energy.  Since we are interested in the limiting case $\eps \rightarrow 0$, we can restrict our analysis without loss of generality to $\eps \in [0, \bar{\eps}]$ for some given fixed $\bar{\eps}>0$. Below we keep track of the nonlinearity constant $k$, as setting $k=0$ reduces the problem to a linear one in which case the imposed conditions on $\aaa$ trivially hold. 
\begin{proposition}\label{Prop:LinStability} Let $\eps \in [0, \bar{\eps}]$ and assumptions \ref{Aone} and \ref{Atwo} on the memory kernel hold. Let $d/2< r \leq p+1$. Assume that the coefficient $\alpha_h \in W^{2, 1}(0,T; V_h)$
 is non-degenerate: there exist constants $\ulal$, $\olal>0$, independent of $h$ and $\eps$, such that
	\begin{equation} \label{non-degeneracy}
	0 < \ulal \leq \aaalin(\alphah)= 1 + k \alphah\leq \olal, \quad  \ \text{for all} \  (x,t) \in \overline{\Om} \times [0,T].
	\end{equation}
	Additionally, let $f \in W^{1,1}(0,T; \Ltwo)$.	Then there exists $m>0$, independent of $h$ and $\varepsilon$, such that if
	\begin{equation} \label{smallness_alphat}
		\vert k \vert \|\alphaht\|_{L^1(\Linf)}+\vert k \vert \|\alphahtt\|_{L^1(\Linf)} \leq m,
			\end{equation}  
then the solution of \eqref{Ph} satisfies the following stability bound:
	\begin{equation} \label{stability_est_lin}
	\begin{aligned}
\sup_{t \in (0,T)}	\left( E_0+ E_1 \right)[\uh](t) \lesssim (E_0+ E_1)[\uh](0)+ \|f\|^2_{W^{1,1}(\Ltwo)},
	\end{aligned}
	\end{equation}
where the hidden constant
does not depend on $h$, $\eps$, or the final time $T$.
\end{proposition}
\begin{proof}
As the linearized problem \eqref{Ph} is non-degenerate thanks to the assumptions on $\alpha_h$, existence of a unique solution \[\uh \in W^{3,1}(0,T; V_h) \hookrightarrow C^{2}([0,T]; V_h)\] follows using the existence theory for Volterra integral equations of second kind or standard ODE theory in case $\frakK=\delta_0$; we provide the details in Appendix~\ref{Appendix:Semi-discrete}. We focus our attention here on obtaining a uniform (in both $\eps$ and $h$) energy bound. \\
\indent	We follow the strategy outlined in \eqref{testing_strategy} and conduct the proof in two steps, corresponding to testing with $\uht$ and testing the time-differentiated problem with $\uhtt$. Note that we rely on the positivity of the kernel assumed in \ref{Atwo} to omit the $\varepsilon$ terms in the estimates below. Testing \eqref{Ph} with $\phi_h=\psiht(s)$ and integrating over $(0,t)$ for $t \in (0,T)$ yields
	\begin{equation}\label{interim_est}
	\begin{aligned}
&	\frac12  \|\sqrt{\aaalin(\alphah)} \psiht(s)\|^2_{\Ltwo} \Big \vert_0^t + 	\frac{c^2}{2}  \|\nabla \psih(s)\|^2_{\Ltwo} \Big \vert_0^t   \\
	\leq& \, \begin{multlined}[t] -\frac12  \inttO \aaalint \psiht^2\dxs+\inttO f \uht \dxs,	\end{multlined} 
	\end{aligned}
	\end{equation}
where $\aaa_t= k \alphaht$. Using Young's inequality, from here we have for any $\gamma_0>0$:
\begin{equation} \label{est_low_1}
	\begin{aligned}
&	\frac12  \|\sqrt{\aaalin(\alphah)} \psiht(s)\|^2_{\Ltwo} \Big \vert_0^t + 	\frac{c^2}{2}  \|\nabla \psih(s)\|^2_{\Ltwo} \Big \vert_0^t   \\
\leq&\,\begin{multlined}[t]\frac12   \|\aaalint\|_{L^1(\Linf)}\|\uht\|^2_{L^\infty(0,t; \Ltwo)}+\frac{1}{4 \gamma_0}\|f\|^2_{L^1(\Ltwo)}\\\hspace*{6cm}+\gamma_0 \|\uht\|^2_{L^\infty(0,t; \Ltwo)}.
\end{multlined}
	\end{aligned}
\end{equation}
If $\gamma_0>0$ and $m>0$ are small enough so that
\[
\frac12   \|\aaalint\|_{L^1(\Linf)}+ \gamma_0 < \frac12 \ulal,
\]
from \eqref{est_low_1} we obtain a lower-order energy bound: 
		\begin{equation} \label{first_est}
		\displaystyle \sup_{t \in (0,T)}	  E_0[\uh](t) \lesssim E_0[\psih](0)+ \|f\|^2_{L^{1}(\Ltwo)}. 
	\end{equation}
We next use the time-differentiated semi-discrete problem:
	\begin{equation} \label{BK_lin_diff}
		\begin{aligned}
		\left(	(\aaalin\uht)_{tt}-c^2\Delta_h \uht-\eps  \frakK*\Delta_h \uhtt, \phi_h\right)_{L^2}= \left(f_t -\eps\eta \frakK(t) \uht(0), \phi_h \right)_{L^2} 
		\end{aligned}
	\end{equation}
	 and test it with $\phi_h=\psihtt(s)$, which after integration in time and exploiting the coercivity property in \ref{Atwo} gives
	\begin{equation}
	\begin{aligned}
	&\frac12  \|\sqrt{\aaalin(\alphah)} \psihtt(s)\|^2_{\Ltwo} \Big \vert_0^t + 	\frac{c^2}{2}  \|\nabla \psiht(s)\|^2_{\Ltwo} \Big \vert_0^t \\
\leq& \, \begin{multlined}[t] -\frac32 \intt (\aaalint \uhtt, \uhtt)_{L^2}\ds - \intt (\aaalintt \uht, \uhtt)_{L^2}\ds\\\hspace*{5cm}+\intt (f_t-\eta \frakK(t) \uht(0) , \uhtt)_{L^2}\ds,
	\end{multlined}
	\end{aligned}
	\end{equation}
where $\aaa_{tt}= k \alphahtt$. We can estimate the last term above using Young's inequality as follows:
\begin{equation}
	\begin{aligned}	
	&\intt (f_t-\eps\eta\frakK(t) \uht(0) , \uhtt)_{L^2}\ds\\
	 \leq&\,( \|f_t\|_{L^1(\Ltwo)}+\eps\eta\|\frakK\|_{\calM(0,T)} \|\uht(0)\|_{\Ltwo)})\|\uhtt\|_{L^\infty(0,t; \Ltwo)}\\
	\leq&\, \begin{multlined}[t]\frac{1}{2 \gamma_1}\|f_t\|^2_{L^1(\Ltwo)} +\frac{1}{2 \gamma_1}\bar{\eps}^2\eta^2 \|\frakK\|^2_{\calM(0,T)}\|\uht(0)\|^2_{\Ltwo}+\gamma_1 \|\uhtt\|^2_{L^\infty(0,t; \Ltwo)}
		\end{multlined}
	\end{aligned}
\end{equation}
for any $\gamma_1>0$. We thus have 
	\begin{equation}\label{ineq_lin_stab}
	\begin{aligned}
		&\frac12  \|\sqrt{\aaalin(\alphah)} \psihtt(s)\|^2_{\Ltwo} \Big \vert_0^t + 	\frac{c^2}{2}  \|\nabla \psiht(s)\|^2_{\Ltwo} \Big \vert_0^t \\
		\leq& \,\begin{multlined}[t]  \frac32 \|\aaalint\|_{L^1(\Linf)} \|\uhtt\|^2_{L^\infty(0,t; \Ltwo)}\\+\frac12\|\aaalintt\|_{L^1(\Linf)}( \|\uht\|_{L^\infty(0,t; \Ltwo)}^2+\|\uhtt\|_{L^\infty(0,t; \Ltwo)}^2)\\+  \frac{1}{2 \gamma_1}\|f_t\|^2_{L^1(\Ltwo)} +\frac{1}{2 \gamma_1}\bar{\eps}^2\eta^2 \|\frakK\|^2_{\calM(0,T)}\|\uht(0)\|^2_{\Ltwo} +\gamma_1 \|\uhtt\|^2_{L^\infty(0,t; \Ltwo)}.
			\end{multlined}
	\end{aligned}
\end{equation}
 For sufficiently small $\gamma_1$ and $m$, adding this inequality to \eqref{interim_est} and proceeding analogously to the first step leads to estimate \eqref{stability_est_lin}, as claimed.	
\end{proof}
For the uniform stability to hold, it remains to prove that the initial energy $(E_0+E_1)\revisedd{R2}{[u_h]}(0)$ is bounded uniformly in $h$ and $\eps$, which boils down to proving uniform boundedness of $\uhtt(0)$. This term should be understood as the solution to
\begin{equation} \label{compat_cond}
	\begin{aligned}
		\begin{multlined}[t]
			((1+k \alphah(0))\uhtt(0), \phi_h)_{L^2}+c^2(\nabla \psih(0), \nabla \phi_h)_{L^2}\\ \hspace*{0.9cm}+\eps( (\frakK*\nabla \psiht)(0), \nabla \phi_h)+(k \alphaht(0)\uht(0), \phi_h)_{L^2} = (f(0), \phi_h)_{L^2}
		\end{multlined}
	\end{aligned}
\end{equation}
for all $\phi_h \in V_h$.  If $\frakK \in L^1(0,T)$, we have $(\frakK*\nabla \psiht)(0)=0$ above because $\psiht$ is $L^\infty$ regular in time (in fact, $\uht \in C^1([0,T]; V_h)$). Recall that we assume $f \in W^{1,1}(0,T; \Ltwo)$, which implies $f \in C([0,T]; \Ltwo)$.
\begin{proposition} \label{Prop:InitialStability} Let the assumption of Proposition~\ref{Prop:LinStability} hold.
	Assume that the exact initial data and source term at time zero are smooth enough in the following sense:
	\begin{equation} \label{smoothness_data}
		\begin{aligned}
			&(u_0, u_1) \in  
					\Hrzero  \times \, \Hrzero, \\[1mm]
			&\utt(0)=\,	(1+ku_0)^{-1}(c^2\Delta u_0+\eps (\frakK*\Delta \ut)(0) +f(0)) \in H^r(\Omega)
		\end{aligned}
	\end{equation}
	with $d/2<  r  \leq p+1$.  Let
	\[
	(\alphah, \alphaht)\vert_{t=0}=(u_{0h}, u_{1h}) = (\Rh u_0, \Rh u_1)
	\]
	and let $\uhtt(0)$ be determined by \eqref{compat_cond}. If the following smallness condition holds:
	\begin{equation} \label{smallness_initial}
|k|\left( 	\Cstint\|u_0\|_{\Linf}+	\Cinv \bar{h}^{r-d/2} (\Capproxritzz +\Capproxint)\|u_{0}\|_{\Hr}  \right) \leq  \frac{1}{2 },
	\end{equation}
then
	\begin{equation}
		\begin{aligned}
			\|\utt(0)-\uhtt(0)\|_{\Ltwo} \leq C(u_0, u_1, k, f(0)) h^r,
		\end{aligned}
	\end{equation}
	where the constant $C>0$ does not depend on $h$ or on $\eps$.
\end{proposition}
\begin{proof}
We begin the proof by noting that error $e_{tt, 0}=\utt(0)-\uhtt(0)$ at time zero satisfies 
	\begin{equation} \label{initial_error}
		\begin{aligned}
			&((1+k u_{0h})e_{tt, 0}, \phi_h)_{L^2}\\
			=&\,\begin{multlined}[t]-k( (u_0-u_{0h})u_{tt}(0), \phi_h)_{L^2} -k((u_1-u_{1h})(u_1+u_{1h}), \phi_h)_{L^2}.
			\end{multlined}
		\end{aligned}
	\end{equation}
	Here we have used the fact that $c^2(\nabla (u_0-u_{0h}), \nabla \phi_h)_{L^2}=0$ due to our choice of the approximate initial data. Similarly, $\eps(\nabla (u_1-u_{1h}), \nabla \phi_h)_{L^2}=0$ when $\frakK=\delta_0$, otherwise this term does not appear. \\
	\indent Equation \eqref{initial_error} is non-degenerate under the assumptions of the statement, on account of the following bound on $u_{0h}$:
	\begin{equation} \label{bound_u0h}
		\begin{aligned}
			&\|u_{0h}\|_{\Linf}\\
			 \leq&\, \|u_{0h}-I_h u_0\|_{\Linf} +\|I_h u_0\|_{\Linf}\\
			\leq&\, \Cinv h^{-d/2} \|u_{0h}-I_h u_0\|_{\Ltwo} +\|I_h u_0\|_{\Linf}  \\
			\leq&\, \Cinv h^{-d/2} (\|u_{0h}- u_0\|_{\Ltwo}+\|u_{0}-I_h u_0\|_{\Ltwo}) +\Cstint\|u_0\|_{\Linf}\\
			\leq&\, \Cinv \bar{h}^{r-d/2} (\Capproxritzz \|u_{0}\|_{\Hr}+\Capproxint\|u_{0}\|_{\Hr}) +\Cstint\|u_0\|_{\Linf}
		\end{aligned}
	\end{equation}	
and the smallness assumption in \eqref{smallness_initial}.  Indeed, by \eqref{bound_u0h} and \eqref{smallness_initial} we have
\[
\frac12 \leq  1+ k u_{0h} \leq \frac32.
\]
We note that $u_0 \in \Hr$  implies $u_0 \in \Linf$ when $r > d/2$. To arrive at the statement, we split the error at time zero into two contributions: \[e_{tt, 0}=(u_{tt}(0)-\Rh u_{tt}(0))-(\uhtt(0) -\Rh u_{tt}(0)):= e_{\Rh \utt(0)}-e_{htt, 0}.\] Then $e_{htt, 0}$ solves
	\begin{equation} \label{initial_error_split}
		\begin{aligned}
			&((1+k u_{0h})e_{htt, 0}, \phi_h)_{L^2}\\
			=&\, 	\begin{multlined}[t] k( (u_0-u_{0h}))u_{tt}(0), \phi_h)_{L^2}+k((u_1-u_{1h})(u_1+u_{1h}), \phi_h)_{L^2}
				\\+	((1+k u_{0h})e_{\Rh \utt(0)}, \phi_h)_{L^2} 	\end{multlined}
		\end{aligned}
	\end{equation}
	for all $\phi \in V_h$. We next choose $\phi_h=e_{htt, 0} \in V_h$. Using Young's inequality for the right-hand side terms immediately yields
	\begin{equation} \label{est_initial_energy}
		\begin{aligned}
		&	\|e_{htt, 0}\|_{\Ltwo} \\
			\lesssim&\, 	\begin{multlined}[t] \|k (u_0-u_{0h})u_{tt}(0)\|_{\Ltwo}+ \|k(u_1-u_{1h})(u_1+u_{1h})\|_{\Ltwo}
			\\	+ \|(1+k u_{0h})e_{\Rh \utt(0)}\|_{\Ltwo}.
			\end{multlined}
		\end{aligned}
	\end{equation}
	It remains to estimate the right-hand side terms. Having in mind that we intend to arrive at $\mathcal{O}(h^r)$ convergence of the approximate solution of the nonlinear problem in the $L^\infty(0,T; L^2(\Omega))$ norm, we should estimate these terms so that $u_0-u_{0h}$ appears in the $\Ltwo$ norm (and likewise $u_1-u_{1h}$ and $e_{\Rh \utt(0)}$).  By H\"older's inequality, 
	\begin{equation} \label{rhs_estimates_zero_time}
		\begin{aligned}
			 \|k (u_{0}-u_{0h})u_{tt}(0)\|_{\Ltwo}
			\leq&\, \vert k \vert \|u_0-u_{0h}\|_{\Ltwo} \|\utt(0)\|_{\Linf}, \\
			 \|k(u_1-u_{1h})(u_1+u_{1h})\|_{\Ltwo} \leq&\, \vert k \vert  \|u_1-u_{1h}\|_{\Ltwo} (\|u_1\|_{\Linf}+\|u_{1h}\|_{\Linf}), 
		\end{aligned}
	\end{equation}
and
	\begin{equation} \label{rhs_estimates_zero_time_}
	\begin{aligned}
		\|(1+k u_{0h})e_{\Rh \utt(0)}\|_{\Ltwo} \leq&\, (1+\vert k \vert\|u_{0h}\|_{\Linf})\| e_{\Rh \utt(0)}\|_{\Ltwo}.
	\end{aligned}
\end{equation}
Analogously to \eqref{bound_u0h}, we have the uniform bound
\begin{equation}
	\begin{aligned}
		\|u_{1h}\|_{L^\infty(\Omega)} \leq \Cinv \bar{h}^{r-d/2} (\Capproxritzz+\Capproxint) \|u_{1}\|_{H^r(\Omega)} +\Cstint\|u_1\|_{L^\infty(\Omega)}.
	\end{aligned}
\end{equation}
	Employing these bounds in \eqref{est_initial_energy} together with the approximation properties of the Ritz projection in the $L^2(\Omega)$ norm leads to the desired estimate.
\end{proof}
\noindent Thanks to the established result,  we can further conclude that
\begin{equation}
	\begin{aligned}
		\| \uhtt(0)\|_{\Ltwo} \leq&\, \| \utt(0)-\uhtt(0)\|_{\Ltwo} +\| \utt(0)\|_{\Ltwo} \leq C(u_0, u_1, k, f(0)) ,
	\end{aligned}
\end{equation}
and thus bound the approximate higher-order initial energy independently of $h$ and $\eps$. 

\subsection{Accuracy of the linearized problem} In the final step of the linear analysis, we determine the accuracy of the finite element approximation.  Let $u$  solve the nonlinear problem \eqref{P_nl} and $\uh$ the linearized problem \eqref{Ph}. The error $e=u-\uh$ satisfies
\begin{equation}
\begin{aligned}
((\aaa(\alphah) e_{t})_t-c^2 \Delta_h e -\eps \frakK*\Delta_h  e_t, \phi_h)_{L^2}
=\, (- (k(u-\alphah)\ut)_t, \phi_h)_{L^2}.
\end{aligned}
\end{equation}
As before, we decompose the error into two parts using the Ritz projection
\begin{equation}
\begin{aligned}
e = (u-\Rh u)-(\uh-\Rh u) := e_{\Rh u}-\eh.
\end{aligned}
\end{equation}
We estimate the second contribution to the error by seeing $e_h$ as the solution to
\begin{equation} \label{eq_error}
	\begin{aligned}
	((\aaa(\alphah) \eht)_t-c^2 \Delta_h \eh -\eps \frakK*\Delta_h  \eht, \phi_h)_{L^2}
	=\, \begin{multlined}[t] (\textup{rhs}, \phi_h)_{L^2}\revisedd{R2}{,}
	\end{multlined}
	\end{aligned} 
	\end{equation}
	with the right-hand side
		\begin{equation} 
\begin{aligned}
&(\textup{rhs}, \phi_h)_{L^2}\\
	 :=&\,( (1+k \alphah) (e_{\Rh u})_{tt}+k \alphaht (e_{\Rh u})_t+k(u-\alphah)\utt+ k(\ut-\alphaht)\ut, \phi_h)_{L^2},
	\end{aligned} 
\end{equation}
and supplemented by the homogeneous boundary and initial conditions (due to our choice of the approximate initial data).  To arrive at \eqref{eq_error}, we have relied on
\begin{equation}
(\frakK *\nabla (u-\Rh u)_t, \nabla \phi_h)_{L^2}=0, \quad \phi_h \in V_h.
\end{equation}
We next employ the derived stability bound and the approximation properties of the Ritz projection to obtain the following result.
\begin{proposition} \label{Prop:Error_est_lin} Let the assumptions of Propositions~\ref{Prop:LinStability} and \ref{Prop:InitialStability} hold. Given $d/2 < r \leq p+1$, we assume that $u \in \usolspace$, defined in \eqref{def_solspace}. 
Then the following error bound holds:
	\begin{equation} \label{Error_est_lin}
	\begin{aligned}
		&\|u-\uh\|_{W^{2, \infty}(\Ltwo)} 
		+ h \| \nabla (u -\uh)\|_{W^{1, \infty}(\Ltwo)}\\[1mm]
		\leq& \, \begin{multlined}[t] \Clin(T) \Big \{(1+\vert k \vert\|\alphaht\|_{L^1(L^\infty(\Omega))}+\vert k \vert\|\alphahtt\|_{L^1(L^\infty(\Omega))})h^{r}\|u\|_{W^{3,1}(\Hr)} \\  \hspace*{4cm}+\vert k \vert\|u-\alphah\|_{W^{2,\infty}(\Ltwo)} \|u\|_{W^{3,1}(L^\infty(\Omega))}  \Big \}, \end{multlined}
	\end{aligned}
	\end{equation}
	where the constant $\Clin$ does not depend on $h$ or on $\eps$. 
\end{proposition} 
	\begin{proof}		
	Using H\"older's inequality and the approximation properties of the Ritz projection, it is not difficult to check that 
		\begin{equation}
	\begin{aligned}
		\|\textup{rhs}\|_{L^1(\Ltwo)} 
		\lesssim&\,_T \begin{multlined}[t] (1+\vert k \vert\|\alphaht\|_{L^1(L^\infty(\Omega))}) h^r \|u\|_{W^{2,1}(\Hr)}\\ \hspace*{3cm}+\vert k \vert\|u-\alphah\|_{W^{1,\infty}(\Ltwo)})\|u\|_{W^{2,1}(\Linf)}. \end{multlined}
	\end{aligned}
\end{equation}	
We next estimate $(\textup{rhs})_t$ by bounding the time derivatives of each of the four summands within $\textup{rhs}$. First we have
\begin{equation}
\begin{aligned}
& \|(1+k \alphaht) (e_{\Rh u})_{tt}\|_{L^1(\Ltwo)} +	\|(1+k \alphah) (e_{\Rh u})_{ttt}\|_{L^1(\Ltwo)} \\
\lesssim_T&\,  \begin{multlined}[t](1+\vert k \vert\|\alphaht\|_{L^1(L^\infty(\Omega))})\|(e_{\Rh u})_{tt}\|_{L^\infty(\Ltwo)} \\\hspace*{4cm}+ (1+\vert k \vert\|\alphah\|_{L^\infty(L^\infty(\Omega))})\|(e_{\Rh u})_{ttt}\|_{L^1(\Ltwo)} 
	\end{multlined}\\
\lesssim_T&\,  (1+\vert k \vert\|\alphaht\|_{L^1(L^\infty(\Omega))}) h^r \|\utt\|_{L^\infty(\Hr)}+ (1+\olal) h^r \|\uttt\|_{L^1(\Hr)}.
 \end{aligned}
\end{equation}
Note that $u \in \usolspace$ implies $\utt \in L^\infty(0,T; H^r(\Omega))$. Similarly,
\begin{equation}
	\begin{aligned}
&\| k \alphahtt (e_{\Rh u})_t\|_{L^1(\Ltwo)}+\| k \alphaht (e_{\Rh u})_{tt}\|_{L^1(\Ltwo)} \\
\lesssim&\, |k| \|\alphahtt\|_{L^1(\Linf)} h^r\|\ut\|_{L^\infty(\Hr)} + |k|\|\alphaht\|_{L^1(\Linf)}h^r \|\utt\|_{L^\infty(\Hr)}.
	\end{aligned}
\end{equation}
Furthermore,
\begin{equation}
\begin{aligned}
\| k\left( (u-\alphah)\utt \right)_t\|_{L^1(\Ltwo)} 
=&\, \|k(\ut-\alphaht)\utt+k(u-\alphah)\uttt\|_{L^1(L^2\Omega))} \\
\lesssim_T&\,|k|\|u-\alphah\|_{W^{1,\infty}(L^2(\Omega))} \|u\|_{W^{3,1}(L^\infty(\Omega))}
\end{aligned}
\end{equation}
and
\begin{equation}
\begin{aligned}
\|k((\ut-\alphaht)\ut )_t\|_{L^1(\Ltwo)} \lesssim_T |k| \|\ut-\alphaht\|_{W^{1,\infty}(\Ltwo)}\|\ut\|_{W^{1,1}(\Linf)}.
\end{aligned}
\end{equation}
Altogether, we infer
		\begin{equation}
		\begin{aligned}
		&\|\textup{rhs}\|_{W^{1,1}(\Ltwo)} \\
		\lesssim_T&\, \begin{multlined}[t]  \Big \{(1+\vert k \vert\|\alphaht\|_{L^1(L^\infty(\Omega))}+\vert k \vert\|\alphahtt\|_{L^1(L^\infty(\Omega))})h^{r}\|u\|_{W^{3,1}(\Hr)} \\  \hspace*{4cm}+\vert k \vert\|u-\alphah\|_{W^{2,\infty}(\Ltwo)} \|u\|_{W^{3,1}(L^\infty(\Omega))}  \Big \}.
			\end{multlined}
		\end{aligned}
		\end{equation}
		Employing stability bound \eqref{stability_est_lin} to the solution of \eqref{eq_error} leads to
		\begin{equation}
		\begin{aligned}
			\sup_{t \in (0,T)}	\left( E_0+E_1  \right) [\eh](t)\lesssim (E_0+E_1)[\eh](0)+ \|\textup{rhs}\|^2_{W^{1,1}(\Ltwo)}.
		\end{aligned}
	\end{equation}
 Due to our choice of initial data, we know that $\eh(0)=\eht(0)=0$, and so
\begin{equation}
	\begin{aligned}
	E_0[\eh](0) = E_0[\uh- \Rh u](0) =0, 
	\end{aligned}
\end{equation}
and by Proposition~\ref{Prop:InitialStability}, we have
\begin{equation}
	\begin{aligned}
		E_1[\eh](0) =&\, \| (\uh-\Rh u)_{tt}(0)\|^2_{\Ltwo}\\
		 \leq&\, 2 \|(\uhtt-\utt)(0)\|^2_{\Ltwo}+ 2\|(u-\Rh u)_{tt}(0) \|^2_{\Ltwo} \leq  C(u_0, u_1, f(0)) h^{2r}.
	\end{aligned}
\end{equation}
Combined with the estimate of $\|\textup{rhs}\|_{W^{1,1}(\Ltwo)}$ derived above, this leads to the bound 
\[
\begin{aligned}
&\|\eh\|_{W^{2, \infty}(\Ltwo)} 
+  \| \nabla \eh\|_{W^{1, \infty}(\Ltwo)}  \\
\leq&\, \begin{multlined}[t]  C \Big \{(1+\vert k \vert\|\alphaht\|_{L^1(L^\infty(\Omega))}+\vert k \vert\|\alphahtt\|_{L^1(L^\infty(\Omega))})h^{r}\|u\|_{W^{3,1}(\Hr)} \\  \hspace*{4cm}+\vert k \vert\|u-\alphah\|_{W^{2,\infty}(\Ltwo)} \|u\|_{W^{3,1}(L^\infty(\Omega))}  \Big \}.\end{multlined}
\end{aligned}
\]
 In the final step of the proof, we use the approximation properties of the Ritz projection to obtain  
\begin{equation}
	\begin{aligned}	
		&\|u-\Rh u\|_{W^{2, \infty}(\Ltwo)} 
	+ h \| \nabla (u -\Rh u)\|_{W^{1, \infty}(\Ltwo)} \lesssim h^r \|u\|_{W^{2,\infty}(\Hr)},
	\end{aligned}	
\end{equation}
 and arrive at the statement.
	\end{proof}
\section{Uniform solvability and accuracy of the nonlinear semi-discrete problem} \label{Sec:NonlinearAnalysis}
In this section, we relate the previous theory to the nonlinear problem \eqref{Ph_nl} via a fixed-point argument. More precisely,  we will rely on Schauder's generalized fixed-point theorem for locally convex spaces (also known as Tikhonov's theorem; see, for example,~\cite[Corollary 9.6]{zeidler1993nonlinear}) to show existence of a fixed-point and then prove uniqueness in an additional step. To this end, we introduce the mapping
\begin{equation}
	\begin{aligned}
		\mathcal{F}: B_h \ni \alpha_h \mapsto u_h,  
	\end{aligned}
\end{equation}
where $\uh$ solves \eqref{Ph} and $\alpha_h$ is taken from a ball in the Banach space \[ X =W^{2, \infty}(0,T; L^2(\Omega)) \cap W^{1, \infty}(0,T; H_0^1(\Omega)),\]
defined by
\begin{equation} \label{ball_Bh}
\begin{aligned}
B_h =\Big\{ \alpha_h \in  W^{2, \infty}(0,T; V_h):\ & 		\|u-\alphah\|_{W^{2, \infty}(\Ltwo)}
+ h \| \nabla (u -\alphah)\|_{W^{1, \infty}(\Ltwo)} \\[0.5mm]  &\leq  C_*h^r\|u\|_{W^{3,1}(\Hr)}, \\[0.5mm]
&(\alphah, \alphaht)_{\vert t=0}=(\Rh u_0, \Rh u_1) \Big\}. 
\end{aligned}
\end{equation}
A fixed point $\alphah=\uh$ will then be a solution of the nonlinear problem and automatically satisfy the error bound. This approach in the combined well-posedness and \emph{a priori} error analysis of the nonlinear problem is inspired by~\cite{makridakis1993finite, ortner2007discontinuous}. The set $B_h$ is non-empty as the Ritz projection $\Rh u$  of the exact solution belongs to it, provided $C_*$ is chosen so that 
\begin{equation}\label{constant_Cstar}
	C_*\geq (C_{W^{3,1}\hookrightarrow W^{2,\infty}}+C_{W^{3,1} \hookrightarrow W^{1,\infty}})\Capproxritz,
\end{equation}
where the two constants in the bracket above are the embedding constants for $W^{3,1}(0,T)\hookrightarrow W^{2,\infty}(0,T)$ and $W^{3,1}(0,T)\hookrightarrow W^{1,\infty}(0,T)$. \\
\indent We will rely on the weak-$*$ topology in the proof of continuity of the mapping below. The reason lies in the fact that the difference  $\bar{u}_h=\uh^{(1)}-\uh^{(2)}$ of two solutions to the semi-discrete problem solves
\begin{equation} \label{diff_eq}
(((1+k \uh^{(1)}) \bar{u}_{ht})_t-c^2\Delta_h \bar{u}_h-\eps  \frakK*\Delta_h \bar{u}_{ht}, \phi_h)_{L^2}= (- (k\bar{u}_h \uht^{(2)})_t, \phi_h)_{L^2}
\end{equation} 
for all $\phi_h \in  V_h$; however, the right-hand side \[- (k\bar{u}_h \uht^{(2)})_t = - k\bar{u}_{ht} \uht^{(2)}- k\bar{u}_h \uhtt^{(2)}\] is insufficiently regular in time to employ the higher-order energy arguments from the previous section. We can only employ the lower-order energy $E_0$ when tackling \eqref{diff_eq}. We note that the space $X$ equipped with the weak-$*$ topology is locally convex. Furthermore,  the ball $B_h$ centered at $u$ is convex and also compact in the weak-$*$ topology in $X$ by the Banach--Alaoglu theorem. 
\begin{theorem}[Asymptotic-preserving \emph{a priori} error estimates]\label{Thm:WellpNl}
Let $T>0$, $\varepsilon \in [0, \bar{\eps}]$, and $h \in (0, \bar{h}]$. Assume that the memory kernel satisfies assumptions \ref{Aone} and \ref{Atwo}. Let the assumptions $(u_0, u_1)$ and the source term $f$ made in Propositions~\ref{Prop:LinStability} and \ref{Prop:InitialStability} hold.  Let $d/2 < r \leq p+1$ and let $u$ be the solution of \eqref{P_nl} in $\usolspace$ that is uniformly bounded with respect to $\eps$ in the $\|\cdot\|_{\usolspace}$ norm. Then there exists $\tilde{m}>0$, such that if
\begin{equation} \label{smallness}
\begin{aligned}
\vert k \vert \left\{\bar{h}^{r-d/2}\|u\|_{W^{3,1}(\Hr)}+\|u\|_{W^{3,1}(\Linf)} \right\} \leq \tilde{m},
\end{aligned}
\end{equation}
then there is a unique semi-discrete solution $\uh$ in the ball $B_h$, defined in \eqref{ball_Bh}, 
\revisedd{R2}{of} the nonlinear semi-discrete problem:
\begin{equation} \label{Ph_nl} \tag{\ensuremath{{\bf {P}}_{h}^{\varepsilon}}}
\left\{ \begin{aligned}
&(((1+k \uh)\uht)_t, \phi_h)_{L^2}+(c^2\nabla \psih+\varepsilon \nabla \frakK*\uht, \nabla \phi_h)_{L^2}= (f, \phi_h)_{L^2}, \\[2mm]
&\text{for all } \phi_h \in V_h \text{ a.e.\ in time}, \\[2mm]
&(\uh, \uht)\vert_{t=0}=	(u_{0h}, u_{1h})\vert_{t=0} =(\Rh u_0, \Rh u_1), 
\end{aligned} \right.
\end{equation}
where the constant $C_*$ \revisedd{R2}{in the definition of} $B_h$ does not depend on $h$ or on $\eps$.
\end{theorem}
\begin{proof}
As announced, we conduct the proof by checking the conditions of the Schauder's fixed-point theorem~\cite[Corollary 9.6]{zeidler1993nonlinear} combined with additionally proving uniqueness. \\
\indent We first prove that $\calF(\Bh) \subset \Bh$.  Take $\alphah \in \Bh$. As we intend to employ the error bound in Proposition~\ref{Prop:Error_est_lin} to conclude that $\mathcal{F}(\alphah) \in B_h$, we first check that $\alphah$ satisfies the assumptions of this proposition. \\
\indent To prove that $1+k\alphah$ does not degenerate, we rely on the properties of the interpolant combined with the inverse inequality in \eqref{inverse_ineq}, similarly to \eqref{bound_u0h}. In this manner, we obtain
\begin{equation} \label{non_degeneracy_alphah}
\begin{aligned}
&\|\alphah(t)\|_{\Linf}\\
\leq&\, \Cinv h^{-d/2}(\|I_h u(t)-u(t)\|_{\Ltwo}+\|u(t)-\alphah(t)\|_{\Ltwo})+\|I_h u(t)\|_{\Linf} 
\end{aligned}
\end{equation}
for all $t \in [0,T]$. Thus, if $\bar{h}$ and $u$ are sufficiently small so that
\begin{equation}
\vert k \vert \left\{\Cinv(\Capproxint+C_*)\bar{h}^{r-d/2}\|u\|_{W^{3,1}(H^r(\Omega))}+\Cstint \|u\|_{L^\infty(L^\infty(\Omega))}\right\} \leq \frac12,
\end{equation}
then 
\begin{equation} \label{non degeneracy}
 \frac12 \leq  1+k\alphah(x,t) \leq \frac32 \qquad \text{for all } (x,t) \in \overline{\Omega} \times [0,T].
\end{equation}
Similarly, we have the following bound: 
\begin{equation} \label{bound_alpha_der}
\begin{aligned}
&\vert k \vert(\|\alphaht\|_{L^1(\Linf)}+\|\alphahtt\|_{L^1(\Linf)})\\
\leq&\, \vert k \vert\left\{\Cinv(\Capproxint+C_*)\bar{h}^{r-d/2}\|u\|_{W^{3,1}(\Hr)}+\Cstint \|u\|_{W^{2,1}(\Linf)}\right\} \\
\lesssim&\, (1+C_*) \tilde{m},
\end{aligned}
\end{equation}
Thus condition \eqref{smallness_alphat} if fulfilled if $\tilde{m}$ is small enough. Therefore, we may apply Proposition~\ref{Prop:Error_est_lin}, which together with using that $\alphah \in B_h$ yields
	\begin{equation} \label{selfmap_est_}
\begin{aligned}
&\|u-\uh\|_{W^{2, \infty}(\Ltwo)} 
+ h \| \nabla (u -\uh)\|_{W^{1, \infty}(\Ltwo)}\\
	\leq& \, \begin{multlined}[t] \Clin(T)\Big \{ (1+\vert k \vert\|\alphaht\|_{L^1(L^\infty(\Omega))}+\vert k \vert\|\alphahtt\|_{L^1(L^\infty(\Omega))})h^{r}\|u\|_{W^{3,1}(\Hr)} \\  \hspace*{4cm}+\vert k \vert\|u-\alphah\|_{W^{2,\infty}(\Ltwo)} \|u\|_{W^{3,1}(L^\infty(\Omega))}\Big \}, \end{multlined} \\
\leq& \,  C(T) h^{r}\left (1+(1+C_* )\tilde{m}+\vert k \vert C_*\|u\|_{W^{3,1}(\Linf)}\right)\|u\|_{W^{3,1}(\Hr)}.
\end{aligned}
\end{equation}
We can further estimate 
\[
\vert k  \|u\|_{W^{3,1}(\Linf)} \leq \tilde{m}
\]
to arrive at
	\begin{equation} \label{selfmap_est}
\begin{aligned}
&\|u-\uh\|_{W^{2, \infty}(\Ltwo)} 
+ h \| \nabla (u -\uh)\|_{W^{1, \infty}(\Ltwo)}\\
\leq& \,  C(T)\left (1+(1+2C_* )\tilde{m}\right)h^{r}\|u\|_{W^{3,1}(\Hr)}.
\end{aligned}
\end{equation}
For the self-mapping property to hold, we then choose $\tilde{m}$ small enough and $C_*$ large enough so that
\begin{equation} \label{selfmap_condition}
	\begin{aligned}
	C(T) \left (1+(1+2C_*\tilde{m})\right) \leq C_*,
	\end{aligned}
\end{equation}
which guarantees that $\uh \in B_h$.
~\\

\noindent{\bf Continuity of the mapping.} To employ the generalized Schauder's theorem, we need to prove continuity of the mapping $\mathcal{F}$ in the weak-$*$ topology. Let $\{\alphah^{(n)}\}_{n \geq 1} \subset B_h$ be a sequence converging to $\alphah \in B_h$ in $X=W^{2, \infty}(0,T; L^2(\Omega)) \cap W^{1, \infty}(0,T; H_0^1(\Omega))$ as $n \rightarrow \infty$.  Then by the self-mapping property, we have 
\begin{equation} \label{sequence_in_Bh}
\uh^{(n)}:=\left \{\mathcal{F} \alphah^{(n)} \right\}_{n \geq 1} \subset \Bh.
\end{equation}
 We aim to prove convergence of the sequence $\{\uh^{(n)}\}_{n \geq 1}$ to $\uh:=\mathcal{F} \alphah$. As $\{\uh^{(n)}\}_{n \geq 1}$ is uniformly bounded in $X$, it has a subsequence (not relabeled) that converges to some $\tilde{u}_h \in \Bh$ in the following sense: 
	\begin{equation} 
\begin{alignedat}{4} 
\uh^{(n)} &\relbar\joinrel\rightharpoonup \tilde{u}_h && \quad \text{ weakly-$*$}  &&\text{ in } &&L^\infty(0,T; H_0^1(\Omega)),  \\
\uht^{(n)} &\relbar\joinrel\rightharpoonup \tilde{u}_{ht} && \quad \text{ weakly-$*$}  &&\text{ in } &&L^\infty(0,T; H_0^1(\Omega)),  \\
\uhtt^{(n)} &\relbar\joinrel\rightharpoonup \tilde{u}_{htt} && \quad \text{ weakly-$*$}  &&\text{ in } &&L^\infty(0,T; \Ltwo),
\end{alignedat} 
\end{equation}
as $n \rightarrow \infty$. Additionally, we have
\begin{equation} 
(1+\alphah^{(n)})\uhtt^{(n)} \relbar\joinrel\rightharpoonup (1+\alphah)\tilde{u}_{htt}  \quad \text{ weakly} \text{ in } L^2(0,T; \Ltwo).
\end{equation}
This follows by the strong convergence of $\{\alphah^{(n)}\}_{n \geq 1}$ in $L^\infty(0,T; L^6(\Omega))$ (due to the continuous embedding $H_0^1(\Omega)$ $\hookrightarrow$ $L^\ell(\Omega)$, $\ell \in [1,6]$) and the uniform bound on $\|\uhtt^{(n)}\|_{L^2(\Lthree)}$:
\begin{equation} \label{unif_bound_uhtt}
	\begin{aligned}
\|\uhtt^{(n)}\|_{L^2(\Lthree)} \leq&\,\CHonethree \|\nabla (\uhtt^{(n)}- \Ih \utt)\|_{L^2(\Ltwo)} + \|\Ih \utt\|_{L^2(\Lthree)}\\
\lesssim&\, \tCinv(\Capproxint+C_*)\bar{h}^{r-1}\|u\|_{W^{3,1}(\Hr)}+\Cstint \|u_{tt}\|_{L^2(\Lthree)},
\end{aligned}
\end{equation}
on account of also \eqref{inverse_ineq2}. Furthermore, $\{\frakK* \uht^{(n)}\}_{n \geq 1}$ is also uniformly bounded:
\[
\|\frakK* \uht^{(n)}\|_{L^2(\Hone)} \leq \|\frakK\|_{\calM(0,T)} \|\nabla \uht^{(n)}\|_{L^2(\Ltwo)}.
\]
Thus we have
	\begin{equation} 
\begin{alignedat}{4} 
\frakK* \uht^{(n)} &\relbar\joinrel\rightharpoonup \frakK* \tilde{u}_{ht} && \quad \text{ weakly}  &&\text{ in } &&L^2(0,T;  H_0^1(\Om)).
\end{alignedat} 
\end{equation}
This is sufficient to pass to the limit in $n$ in the equation solved by $\uh^{(n)}$ and prove that $\tilde{u}_h$ solves
 \begin{equation} 
\begin{aligned}
 (((1+k \alphah)\tilde{u}_{ht})_t, \phi_h)_{L^2}+(c^2\nabla \tilde{u}_h+\varepsilon  \frakK*\nabla \tilde{u}_{ht}, \nabla \phi_h)_{L^2}= (f, \phi_h)_{L^2}
 \end{aligned} 
 \end{equation}
 for all $\phi_h \in V_h$ a.e.\ in time. The attainment of approximate initial conditions follows by the definition of $B_h$. Then by the uniqueness of solutions to \eqref{Ph} and a subsequence-subsequence argument, we conclude that the whole sequence converges to $\tilde{u}_h \in B_h$ and that this limit must be equal to $\uh= \mathcal{F}(\alphah)$. \\
 \indent Schauder's theorem thus guarantees that there exists a fixed point of the mapping in $B_h$, which provides us with a solution of the nonlinear semi-discrete problem. \\

\noindent{\bf Uniqueness of the semi-discrete solution.} Let us assume that there are at least two fixed points $\uh^{(1)}$ and $\uh^{(2)}$ in $\Bh$. The difference $\bar{u}_h=\uh^{(1)}-\uh^{(2)}$ would solve \eqref{diff_eq}
with homogeneous initial data. Testing with $\bar{u}_{ht}$ and proceeding similarly to the proof of the lower-order linear stability bound \eqref{first_est} with the right-hand side $ -(k\bar{u}_h \uht^{(2)})_t $ leads to
\begin{equation}
\begin{aligned}
 &\|\bar{u}_{ht}(t)\|^2_{\Ltwo}+\|\nabla \bar{u}_{h}(t)\|^2_{\Ltwo}\\
  \lesssim&\, \vert k \vert \| \bar{u}_{ht} \uht^{(2)}+\bar{u}_h \uhtt^{(2)}\|^2_{L^1(0,t;\Ltwo)} \\
 \lesssim&\,  \vert k \vert \| \bar{u}_{ht}\|^2_{L^2(0,t;\Ltwo)} \|\uht^{(2)}\|^2_{L^2(\Linf)}+\|\bar{u}_h\|_{L^2(0,t;\Lsix)}^2\| \uhtt^{(2)}\|^2_{L^2(\Lthree)}. 
\end{aligned}
\end{equation}
Provided we have uniform bounds on $\|\uht^{(2)}\|_{L^2(\Linf)}$ and $\| \uhtt^{(2)}\|_{L^2(\Lthree)}$, from here we can employ Gr\"onwall's inequality to conclude that $\bar{u}_h=0$. The latter is given in \eqref{unif_bound_uhtt}. The bound on  $\|\uht^{(2)}\|_{L^2(\Linf)}$ can be obtained analogously to \eqref{non_degeneracy_alphah}. Indeed, we have
\begin{equation}
\|\uht^{(2)}\|_{L^2(\Linf)} \leq \Cinv(\Capproxint+C_*)\bar{h}^{r-d/2}\|u\|_{W^{3,1}(\Hr)}+\Cstint \|u_t\|_{L^2(\Linf)}.
\end{equation}
Thus the semi-discrete solution is unique, which concludes the proof.
\end{proof}
We note that using Agmon's interpolation inequality~\cite[Lemma 13.2, Ch.\ 13]{agmon2010lectures}: 
	\begin{equation}\label{unif_bound_energynorm}
\begin{aligned}
\|v\|_{\Linf} \leq C_{\textup{A}}\|v\|_{\Ltwo}^{1-d/4}\|v\|^{d/4}_{H^2(\Omega)}, \quad v \in H^2(\Omega) \cap H_0^1(\Omega), 
\end{aligned}
\end{equation}
the smallness assumption in \eqref{smallness} can be replaced by
\begin{equation} \label{smallness_new}
\begin{aligned}
\vert k \vert \left\{\bar{h}^{r-d/2}\|u\|^{1-d/4}_{W^{3,1}(\Hr)}+\|u\|^{1-d/4}_{W^{3,1}(\Ltwo)} \right\}\|u\|^{d/4}_{W^{3,1}(\Hr)} \leq \tilde{m},
\end{aligned}
\end{equation}
with $2 \leq r \leq p+1$. This smallness condition is further mitigated in practice by the fact that $|k|$ is relatively small as it is inversely proportional to the sound of speed squared; see, e.g.,~\cite[Ch.\ 5]{kaltenbacher2007numerical}.  \\
\indent We also state a corollary of the previous theory that will be useful in studying the limiting behavior (and obtaining the rate of convergence) of the semi-discrete solution as $\eps \rightarrow 0$.
\begin{corollary} \label{Corollary} Under the assumptions of Theorem~\ref{Thm:WellpNl}, the solution of \eqref{Ph_nl} satisfies
	\begin{equation}
		\|\uht\|_{L^\infty(\Linf)} +\reviseV{\|\nabla \uhtt\|_{L^2(\Ltwo)}} \leq C\|u\|_{W^{3,1}(\Hr)}, 
	\end{equation}
	where $C>0$ does not depend on $h$ or $\eps$. Furthermore,
	\begin{equation}
		 \tfrac12 \leq  1+k\uh(x,t) \leq \tfrac32 \qquad \text{for all } (x,t) \in \overline{\Omega} \times [0,T].
	\end{equation}
\end{corollary}
\section{Limiting behavior of semi-discrete solutions in the zero dissipation limit} \label{Sec:EpsLimit}
We have determined the conditions under which the finite-element semi-discretization is stable and accurate independently of the perturbation parameter $\eps$. We next discuss the limiting behavior of the perturbed semi-discrete problem as $\eps \rightarrow 0$, under the assumptions of Theorem~\ref{Thm:WellpNl}.\\
\indent Let $\{\uh^{\varepsilon}\}_{\varepsilon \in (0, \bar{\eps}]}$ be a family of solutions of the perturbed semi-discrete problem \eqref{Ph_nl} and let $\uh^{0}$ be the solution of the unperturbed semi-discrete problem given by
\begin{equation} \label{Ph_nl_zero} \tag{\ensuremath{{\bf {P}}_{h}^{0}}}
	\left\{ \begin{aligned}
		&(((1+k \uh^0)\uht^0)_t, \phi_h)_{L^2}+(c^2\nabla \psih^0, \nabla \phi_h)_{L^2}= (f, \phi_h)_{L^2}, \\[2mm]
		&\text{for all } \phi_h \in V_h \text{ a.e.\ in time},\\[2mm]
		&(\uh^0, \uht^0)\vert_{t=0} =(\Rh u_0, \Rh u_1).
	\end{aligned} \right.
\end{equation}
The difference ${w}_h^{\eps}= \uh^{\eps}-\uh^{0} $ solves
\begin{equation}\label{diff_eps_limit}
	(((1+k \uh^{\eps}) w^{\eps}_{ht})_t-c^2\Delta_h w^{\eps}_h+ k(w^{\eps}_h \uht^{0})_t, \phi_h)_{L^2}=(\eps \frakK* \Delta_h u^{\eps}_{ht}, \phi_h)_{L^2}
\end{equation}
{for all $\phi_h \in V_h$, with homogeneous initial data. We will test this problem with $w_{ht}^\eps$ to obtain an $\eps$-limiting result.
	\begin{theorem} \label{Thm:epsConv_lower} Let the assumptions of Theorem~\ref{Thm:WellpNl} hold. Then for fixed $h \in (0, \bar{h}]$ determined by Theorem~\ref{Thm:WellpNl}, the family $\{\uh^\eps\}_{\eps \in (0, \bar{\eps}]}$ of solutions to \eqref{Ph_nl} converges to the solution $\uh^0$ of \eqref{Ph_nl_zero} as $\eps \rightarrow 0$ in the following sense:
		\begin{equation}
			\begin{aligned}
				\|\uh^\eps-\uh^0\|_{L^\infty(\Ltwo)} +	\|\uht^\eps-\uht^0\|_{L^\infty(\Ltwo)}+	\|\nabla(\uh^\eps-\uh^0)\|_{L^\infty(\Ltwo)}\leq C \eps,
			\end{aligned}
		\end{equation}
		where the constant $C>0$ does not depend on $h$ or $\eps$.
	\end{theorem}
	\begin{proof}
		As announced, we test \eqref{diff_eps_limit} with $w_{ht}^\eps$ and proceed analogously to the proof of uniqueness in Theorem~\ref{Thm:WellpNl}. The new term here is the convolution term on the right-hand side. To treat it, we first  integrate by parts in time:
		\begin{equation}
			\begin{aligned}
				&-	\intt \eps( \frakK*\nabla u^{\eps}_{ht}, \nabla w^\eps_{ht})_{L^2} \ds \\
				=&\, - \eps( (\frakK*\nabla u^{\eps}_{ht} )(t), \nabla w^\eps_{h}(t))_{L^2}  +	\intt \eps( \frakK*\nabla u^{\eps}_{htt}+ \eta \frakK(t) \nabla u_{h1}, \nabla w^\eps_{h})_{L^2} \ds
			\end{aligned}
		\end{equation}		
	for $t \in [0,T]$,	where we have also used formula \eqref{derivative_convolution} and the fact that $w_{ht}^\eps(0)=0$. Recall that thanks to Theorem~\ref{Thm:WellpNl} and Corollary~\ref{Corollary}, we have the uniform bound 
	\begin{equation} \label{unif_bound}
	\|\nabla u^{\eps}_{ht}\|_{L^\infty(\Ltwo)}+\|\nabla \uhtt^\eps\|_{L^2(\Ltwo)} \leq C.
	\end{equation}
	 By H\"older's inequality and Young's convolution inequality, we find that	
		\begin{equation}
			\begin{aligned}
				& \left|	\intt \eps( \frakK*\nabla u^{\eps}_{ht}, \nabla w^\eps_{h})_{L^2} \ds \right | \\
				\leq&\, \begin{multlined}[t] \eps \|\frakK\|_{\calM(0,T)}\|\nabla u^{\eps}_{ht}\|_{L^\infty(\Ltwo)} \|\nabla w^\eps_{h}\|_{L^\infty(0,t;\Ltwo)}\\\hspace*{2.5cm}+ \eps \|\frakK\|_{\calM(0,T)} \|\nabla \uhtt^\eps\|_{L^2(\Ltwo)} \|\nabla w_h^\eps\|_{L^2(0,t;\Ltwo)}\\\hspace*{3.2cm}+ \eps \eta \|\frakK\|_{\calM(0,T)} \|\nabla u_{h1}\|_{\Ltwo} \|\nabla w_h^\eps\|_{L^\infty(0,t;\Ltwo)}.
				\end{multlined}
			\end{aligned}
		\end{equation}		
		An application of Young's inequality and using \eqref{unif_bound} lead to 
		\begin{equation}
			\begin{aligned}
				& \left|	\intt \eps( \frakK*\nabla u^{\eps}_{ht}, \nabla w^\eps_{h})_{L^2} \ds \right | \\
				\lesssim&\, \begin{multlined}[t] \eps^2 \|\frakK\|_{\calM(0,T)}^2 (1+ \|\nabla u_{h1}\|^2_{\Ltwo} ) +\gamma \|\nabla w_h^\eps\|^2_{L^\infty(0,t;\Ltwo)}+\|\nabla w_h^\eps\|^2_{L^2(0,t;\Ltwo)}
				\end{multlined}
			\end{aligned}
		\end{equation}		
		for any $\gamma>0$. Choosing $\gamma$ sufficiently small (independently of $h$ and $\eps$)  and otherwise proceeding as in the proof of uniqueness in Theorem~\ref{Thm:WellpNl} results in the bound
		\[
		\|\uht^\eps-\uht^0\|^2_{L^\infty(\Ltwo)}+	\|\nabla(\uh^\eps-\uh^0)\|^2_{L^\infty(\Ltwo)}\leq C \eps^2.
		\]
		Additionally relying on the fact that
		\[
		\|\uh^\eps-\uh^0\|_{L^\infty(\Ltwo)} \leq T  \|\uht^\eps-\uht^0\|_{L^\infty(\Ltwo)}
		\]
		completes the proof.
	\end{proof}
}
Theorem~\ref{Thm:epsConv_lower} \reviseV{mimics the $\eps$-limiting result in the continuous setting obtained in~\cite[Corollary 4.2]{kaltenbacher2022limiting}. It} establishes the conditions under which  semi-discrete solutions of \eqref{Ph_nl} preserve the limiting behavior of the exact solution as $\eps \rightarrow 0$. Additionally, it provides the error made when replacing \eqref{Ph_nl_zero} by \eqref{Ph_nl}. The latter information might prove to be useful in cases where artificial dissipation is added to the undamped problem. This manner of reducing non-physical dispersion effects is used in the simulation of nonlinear ultrasound phenomena; see, for example,~\cite[Section 3]{schenke2022explicit}.  The expected error in, for example, $L^\infty(0,T; \Ltwo)$ norm, is then
\[\
\begin{aligned}
	\|u^0-\uh^\eps\|_{L^\infty(\Ltwo)} \leq \|u^0-\uh^0\|_{L^\infty(\Ltwo)} +\|\uh^0-\uh^\eps\|_{L^\infty(\Ltwo)}  \leq C_1 h^r  +C_2 \eps.
\end{aligned}
\]
\section{Numerical experiments} \label{Sec:Numerics}
To conclude, we illustrate theoretical considerations numerically in one- and two-dimensional settings.

\begin{exmp}\label{FirstExample}
	\end{exmp} We first take the spatial domain to be $\Omega=(0, 0.5)$ and choose the right-hand side $f$ in the inviscid problem \eqref{P_nl_zero} such that the exact solution is given by
\begin{equation}
	u^0(x,t) = \sin(4 \pi x) \cos(4 \pi t).
\end{equation}
The initial conditions are then set to 
\begin{equation}
	u_0(x) = \sin(4 \pi x), \qquad u_1(x)=0.
\end{equation}
We choose the medium parameters as follows:
\begin{equation} \label{medium_parameters}
	c=1500, \quad k=-3 \cdot 10^{-9}.
\end{equation}
 Discretization in space is performed with continuous piecewise linear finite elements and implemented using MATLAB.  For time stepping, we use a Newmark predictor-corrector scheme in the mass effective form with the choice of the parameters $(\frac14, \frac12)$; see~\cite[Ch.\ 5]{kaltenbacher2007numerical} for details.\\
\indent In case of the time-fractional evolution, we follow the approach of~\cite{kaltenbacher2022fractional} and combine the Newmark method with an $L^1$-type scheme for the fractional derivative. The latter relies on the following approximation of the derivative of order $\alpha$ of $\uh^\eps$ at the time step $t_n$:
\begin{equation}
\begin{aligned}
\textup{D}_t^\alpha \uh^\eps(t_n) =&\, \frac{1}{\Gamma(1-\alpha)} \sum_{j=0}^{n-1}\int_{t_j}^{t_{j+1}} (t_n-s)^{-\alpha}\uht^\eps(s) \ds \\
\approx&\, \frac{1}{\Gamma(1-\alpha)} \sum_{j=0}^{n-1}\frac{\uht^\eps(t_{j+1})+\uht^\eps(t_j)}{2}\int_{t_j}^{t_{j+1}} (t_n-s)^{-\alpha} \ds \\
=&\, \frac{1}{2\Gamma(2-\alpha)} \sum_{j=0}^{n-1}\left(\uht^\eps(t_{j+1})+\uht^\eps(t_j)\right)\left((t_n-t_j)^{1-\alpha}-(t_n-t_{j+1})^{1-\alpha}\right);
\end{aligned}
\end{equation}
we refer to~\cite[Sec.\ 2.1]{kaltenbacher2022fractional} and~\cite{jin2019numerical} for details. \revisedd{R2}{We mention that an alternative approach to time stepping based on the trapezoidal rule and A-stable convolution quadrature within the framework of~\cite{lubich1986discretized, lubich1988convolution} has been developed and analyzed in \cite{baker2022numerical}.} The nonlinearity is resolved through a fixed-point iteration with the tolerance set to $10^{-8}$. The final time is set to $T=0.25$. \vspace*{2mm}

\noindent  We first conduct a sequence of simulations of \eqref{Ph_nl} for
\begin{equation}
\begin{aligned}
\eps \in \{0.5, 1, 2, 4, 8\} \times 10^{-6}.
\end{aligned}
\end{equation}
The spatial discretization parameter $h$ and time step $\Delta t$ are fixed beforehand, independently of $\eps$. We consider two cases: $\frakK$ being the Dirac delta $\delta_0$ and the Abel kernel \eqref{Abel_krenel} with $\alpha=0.6$ and compute the difference
\begin{equation}
\begin{aligned}
e_\eps=|\|u^0-\uh^\eps\|_{L^\infty(\Ltwo)}-\|u^0-\uh^0\|_{L^\infty(\Ltwo)}|.
\end{aligned}
\end{equation} 
\revisedd{R2}{Other values of $\alpha \in (0,1)$ resulted in the same empirical behavior, so we do not report the results of those experiments here}. By the previous theory, we know that
\begin{equation}
\textup{error}_\eps \leq \|\uh^0-\uh^\eps\|_{L^\infty(\Ltwo)} \leq C\eps.
\end{equation}
Experimentally obtained order of convergence is given in Table~\ref{Table1}, showing agreement with the predicted order $1$.\\
\begin{center}
	\begin{tabular}[h]{|c|c|c|}
		\hline
		&&\\[-3mm]
		$\eps$ &   $\textup{error}_\eps$ with $\frakK=\delta_0$  &   $\textup{error}_\eps$ with $\alpha=0.6$  \\[5pt]
		\hline\hline	
		$8 \times 10^{-6}$ &  -- & --\\[1mm]
			$4 \times 10^{-6}$ &    1.0006 & 0.9993\\[1mm]
			$2 \times 10^{-6}$ &  0.9986	& 1.0002\\[1mm]
		$1 \times 10^{-6}$ &1.0002 & 1.0016 \\[1mm]
			$0.5 \times 10^{-6}$ & 1.0079 & 0.9968 \\
		\hline
	\end{tabular}
	~\\[3mm]
	\captionof{table}{Experimental order of convergence with respect to $\eps$ for fixed $h$ in Example~\ref{FirstExample}} \label{Table1}  
\end{center} 
\noindent  We further perform a sequence of simulations of \eqref{Ph_nl} for $h = \frac{1}{2^i}$, $i \in \{6, \ldots, 10\}$ with
$\eps=h^2$
and $\Delta t << h$ fixed. The experimental order of convergence for
\[
\textup{error}_h=\|u^0-\uh^\eps\|_{L^\infty(\Ltwo)}, \quad \eps=h^2
\]
is reported in Table~\ref{Table2}. As the orders for $\frakK=\delta_0$ and the Abel kernel \eqref{Abel_krenel} with  $\alpha=0.6$ were the same, we only state the values for the latter. \\
\begin{center}
	\begin{tabular}[h]{|c|c|}
		\hline
		&\\[-3mm]
		$h$ &   $\textup{error}_h$ with $\alpha=0.6$  \\[5pt]
		\hline\hline	
	 1/64 &  -- \\[1mm]
	 1/128 &  1.9997\\[1mm]
	 1/256 & 1.9604	 \\[1mm]
	 1/512 &1.9038 \\[1mm]
	 1/1024 & 1.9841 \\
		\hline
	\end{tabular}
	~\\[3mm]
	\captionof{table}{Experimental order of convergence with respect to $h$ with $\eps=h^2$ in Example~\ref{FirstExample}} \label{Table2}  
\end{center} 
We see that the experimentally obtained order in Table~\ref{Table2} agrees with the theory in case $\eps=h^2$ and $r=2$:
\[
\textup{error}_h \leq \|u^0-\uh^0\|_{L^\infty(\Ltwo)}+\|\uh^0-\uh^\eps\|_{L^\infty(\Ltwo)} \leq C h^2.
\]

\begin{exmp}\label{SecondExample}
\end{exmp} We additionally numerically investigate asymptotic behavior of the discrete Westervelt equation in an application-oriented setting that is outside of our theory (due to the type of boundary excitation used) and in which the exact solution is not known. We take the spatial domain to be $\Omega=(0, 0.3)^2$ and excite the wave using Neumann boundary conditions:
\begin{equation}
	\begin{aligned}
	\frac{\partial u}{\partial n}=	g(t)= \begin{cases}
		\	\mathcal{A}\sin(w t), \ &\text{for }  x=0,\ 0.21 \leq y \leq 0.3, \\
		\	0, \ &\text{elsewhere on } \partial \Omega,
		\end{cases}
	\end{aligned}
\end{equation}
with the amplitude and angular frequency set to $\mathcal{A=}2 \cdot 10^3$ and $w=6 \pi \cdot 10^4 $, respectively. The initial conditions and source term are set to zero. The medium parameters are chosen as before in \eqref{medium_parameters} . We take here $\frakK$ to be the Dirac delta distribution $\delta_0$, resulting in the $-\eps \Delta_h \uht^\eps$ term in the simulated problem. \\
\indent Time stepping is carried out with the Newmark scheme with the parameters set to $(0.45, 0.75)$. This choice is common with high-frequency acoustic simulation; see, e.g.,~\cite[Section 7]{muhr2017isogeometric}. The final time is $T=1.4 \cdot10^{-4}$. Experiments are performed in FEniCSx~\cite{alnaes2014unified, alnaes2015fenics}, using continuous piecewise linear finite elements on triangles. Figure~\ref{Figure:snapshots} provides two snapshots of the ultrasonic wave as it propagates. We can observe steepening of the wave front, counteracted by the spreading of the wave into the domain. \\

\begin{figure}[h]
	\centering
	\begin{subfigure}{.5\textwidth}
		\centering
		\includegraphics[scale=0.35]{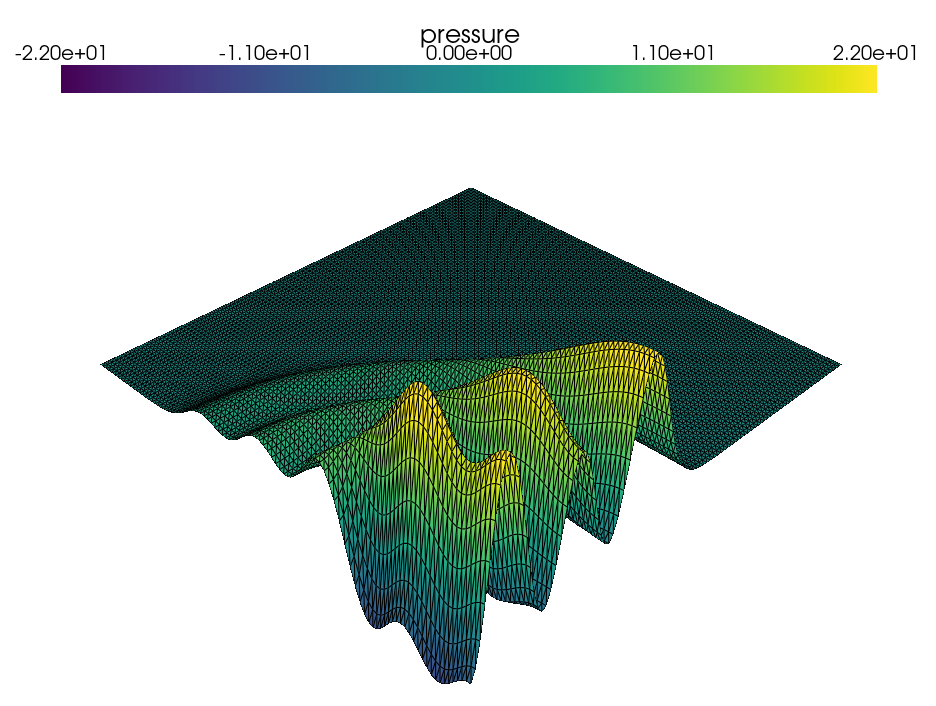}
		\caption{$t=t_1$}
	\end{subfigure}%
	\begin{subfigure}{.5\textwidth}
		\centering
	\raisebox{0.1cm}{\includegraphics[scale=0.35]{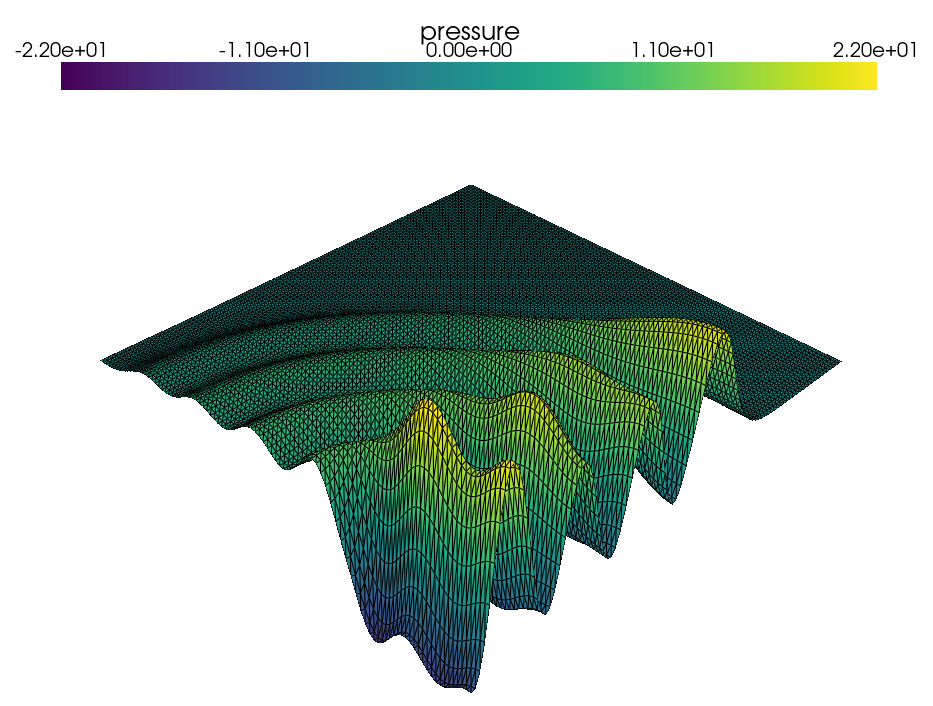}}
		\caption{$t=t_2$}
	\end{subfigure}
	\caption{Snapshots of the sound wave in Example~\ref{SecondExample} with $\eps=10^{-6}$}
	\label{Figure:snapshots}
\end{figure}
\begin{figure}[h]
	\centering
	\includegraphics[scale=0.5]{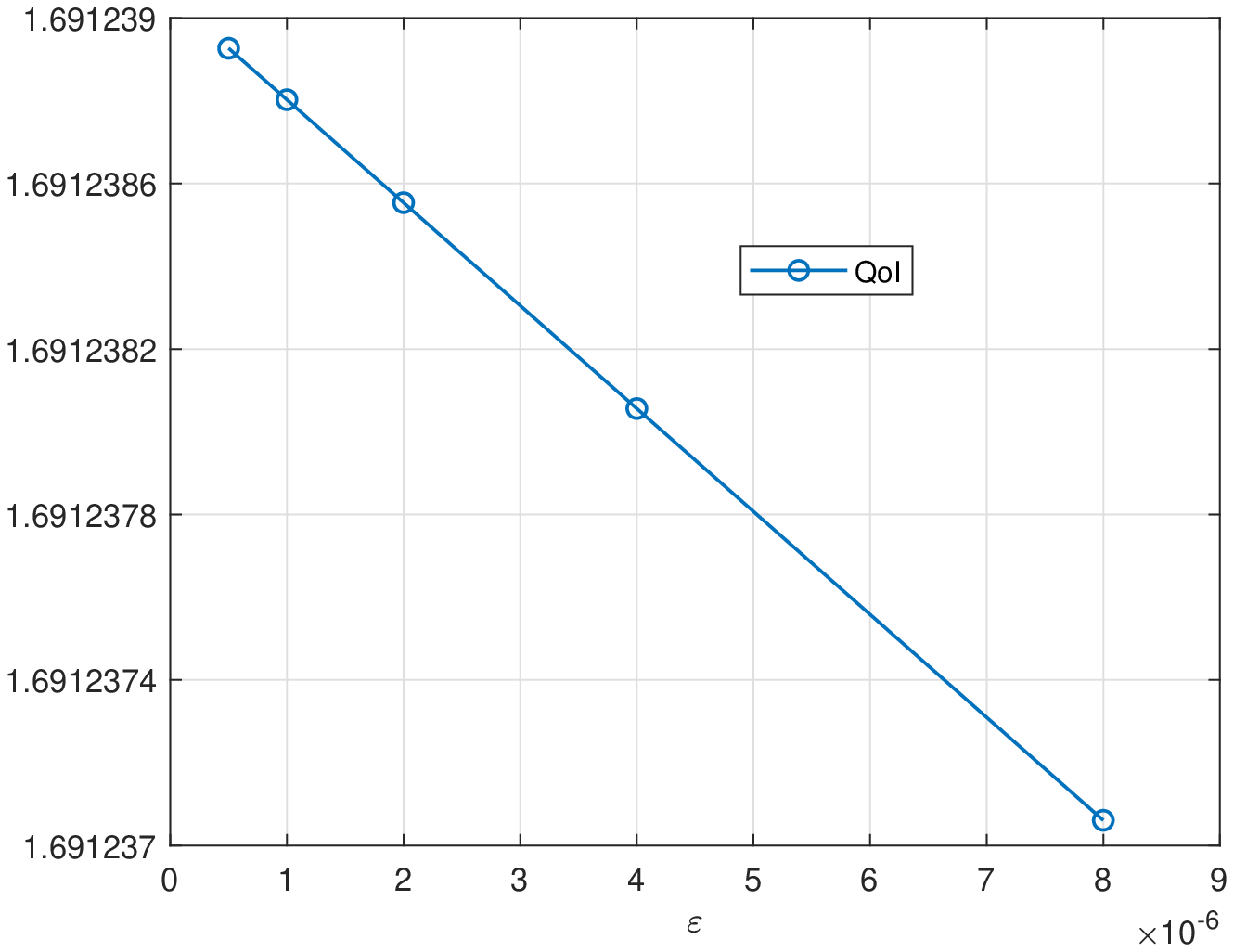}
	\caption{Plot of QoI$=\|\uh^\eps\|_{L^\infty(\Ltwo)}$  in Example~\ref{SecondExample} for different values of $\eps$ and fixed $h$}
	\label{Fig:ConvEps2D}
\end{figure}

\indent Since in this setting we do not have the exact solution at our disposal, we compute and plot a quantity of interest given by the $L^\infty(0,T; \Ltwo)$ norm of the pressure field:
\[
\text{QoI} = \|u_h^{\eps}\|_{L^\infty(\Ltwo)}.
\]
We fix the spatial and temporal discretization and vary the damping parameter: \[\eps \in \{1, 2, 4, 8\} \times 10^{-6}.\] Figure~\ref{Fig:ConvEps2D} shows linear dependence of the quantity of interest on this parameter, indicating that the $\eps$ asymptotics extends to this setting as well.

\section*{Conclusion and outlook}
In this work, we have determined sufficient conditions under which finite element discretizations of Westervelt-type equations with strong or time-nonlocal damping preserve the vanishing dissipation asymptotics. To this end, we have devised a testing strategy that led to an $\eps$-uniform energy bound for the semi-discrete solution, at first for a linearized problem and then, via a fixed-point argument, for the nonlinear one.  Under suitable smoothness and smallness conditions (see \eqref{smallness_new}), the resulting \emph{a priori} estimates remain valid as $\eps \rightarrow 0$, thus allowing to use the same discretization for the perturbed and unperturbed problems.  We have furthermore established the rate of convergence with respect to the perturbation parameter $\eps$, which is particularly informative in the context of introducing artificial diffusion to the inviscid problem.  \\
\indent Future work will be concerned with the analysis of fully discrete schemes in the asymptotic-preserving regime as well as considering enhanced models of nonlinear acoustics. In particular, the analysis of the Newmark algorithm used in the experiments would be challenging but worthwhile, where the results for the linear wave equation in~\cite{kaltenbacher2022fractional} could serve as a starting point. It is also of interest to investigate asymptotic behavior of the (semi-)discrete acoustic  models with quadratic gradient nonlinearities present in the nonlocal Kuznetsov and Blackstock equations~\cite{kaltenbacher2023limiting}.
\section*{Acknowledgments}  The author is thankful to Mostafa Meliani (Radboud University) for the careful reading of the manuscript and helpful comments. 

%% file: appendix.tex
\begin{appendices} 
	\section{Solvability of the linear semi-discrete problem}	\label{Appendix:Semi-discrete}
	We present here the proof of existence of a unique solution of \eqref{Ph_nl} under the assumptions of Proposition~\ref{Prop:LinStability}. We only discuss the more involved case of the nonlocal equation with $\frakK \in L^1(0,T)$, as otherwise standard ODE theory can be used; see, e.g.,~\cite[Section 3]{kaltenbacher2022parabolic}. Given a basis $\{\phi_i\}_{i=1}^{n=n(h)}$ of $V_h$, we have
	\begin{equation}
	\begin{aligned}
	\uh(x,t) =\sum_{i=1}^{n=n(h)} \xi_i(t)\phi_i(x).
	\end{aligned}
\end{equation}
Denoting $\bfxi= [\xi_1 \ldots \xi_n]^T$, the semi-discrete problem can be rewritten in matrix form as follows:
	\begin{equation} \label{matrix_eq}
		\begin{aligned}
			\mathbb{M}_{1+k\alphah}(t)\boldsymbol{\xi}_{tt}+ c^2 \mathbb{K} \boldsymbol{\xi}+\eps \mathbb{K}\, \frakK*\boldsymbol{\xi}_t+	\mathbb{M}_{\alphaht}(t)\boldsymbol{\xi}_{t}= \boldsymbol{b}(t),
		\end{aligned}
	\end{equation}
	where the entries of the weighted mass matrices $\mathbb{M}_{1+k\alphah}(t)=[\mathbb{M}_{1+k\alphah, ij}]$, $\mathbb{M}_{\alphaht}(t)=[\mathbb{M}_{\alphaht, ij}]$,  and  the stiffness matrix $\mathbb{K}=[\mathbb{K}_{ij}]$ are given by
	\begin{equation} \label{matrices}
		\begin{aligned}
			 &\mathbb{M}_{1+k\alphah, ij}(t)= ((1+k \alphah(t))  \phi_i, \phi_j)_{L^2}, \quad  \mathbb{M}_{\alphaht, ij}(t)= (k \alphaht(t)  \phi_i, \phi_j)_{L^2},  \\[1mm]
			  &\mathbb{K}_{ij}= (\nabla \phi_i, \nabla \phi_j)_{L^2}.
		\end{aligned}
	\end{equation}
The entries of the right-hand side vector $\boldsymbol{b}=[\boldsymbol{b}_j]$ are
\[
\boldsymbol{b}_j(t)=(f(t), \phi_j)_{L^2}.
\]
Note that since $f \in W^{1,1}(0,T; \Ltwo) \hookrightarrow C([0,T]; \Ltwo)$, we have $\boldsymbol{b} \in C[0,T]$.	If	we  introduce the vectors of coordinates of the approximate initial data $(u_{0h}, u_{1h})$ in the basis: 
	\[
	\boldsymbol{\xi}_0=[\xi_{0,1} \ \xi_{0,2} \ \ldots \ \xi_{0,n}]^T, \quad \boldsymbol{\xi}_1=[\xi_{1,1} \ \xi_{1,2} \ \ldots \ \xi_{1,n}]^T,
	\]
	then by setting $\bfmu=\boldsymbol{\xi}_{tt}$, we have
	\begin{equation} \label{eq_xi}
		\bfxi_t(t)=1\Lconv\bfmu+\bfxi_1, \quad	\boldsymbol{\xi}(t)=\boldsymbol{\xi}_0 +t \boldsymbol{\xi}_1+1*1*\bfmu.
	\end{equation}
	Therefore, the semi-discrete problem can be rewritten as
	\begin{equation}
		\begin{aligned}
			\begin{multlined}[t]
				\mathbb{M}_{1+k\alphah}(t)\bfmu+ c^2 \mathbb{K}(\boldsymbol{\xi}_0 +t \boldsymbol{\xi}_1+1*1*\bfmu)+\eps \mathbb{K}\, \frakK*(1\Lconv\bfmu+\bfxi_1)\\ \hspace*{4cm}
				+\mathbb{M}_{\alphaht}(t)(1\Lconv\bfmu+\bfxi_1)=\ \boldsymbol{b}(t). \end{multlined}
		\end{aligned}
	\end{equation}
	Since $\mathbb{M}_{1+k\alphah}$ is positive definite due to the assumptions on $\alphah$, the semi-discrete problem can be seen as a system of Volterra integral equations 
	\begin{equation}
		\begin{aligned}
			\bfmu+\boldsymbol{K}*\bfmu(s)\ds = \boldsymbol{{f}}(t)
		\end{aligned}
	\end{equation}
	with the kernel
	\begin{equation}
	\begin{aligned}
		\boldsymbol{K}= \left\{\mathbb{M}_{1+k\alphah}(t)\right\}^{-1}\left\{c^2 \mathbb{K} 1*1+\eps \mathbb{K}\, \frakK*1+\mathbb{M}_{\alphaht}(t)1\right\}
	\end{aligned}	
		\end{equation}
	and the right-hand side
	\begin{equation}
	\begin{aligned}
	\boldsymbol{{f}}(t)=&\, \begin{multlined}[t]\left\{\mathbb{M}_{1+k\alphah}(t)\right\}^{-1}\left\{\boldsymbol{b}(t)- c^2 \mathbb{K}(\boldsymbol{\xi}_0 +t \boldsymbol{\xi}_1)-\eps \mathbb{K}\, \frakK*\bfxi_1-\mathbb{M}_{\alphaht}(t)\bfxi_1 \right\}.
	\end{multlined}
	\end{aligned}	
\end{equation}
Since $\alphah \in W^{2,1}(0,T; V_h) \hookrightarrow C^1([0,T]; V_h)$ and $\boldsymbol{b} \in C[0,T]$, we have
	\[
	\boldsymbol{K} \in C[0,T], \quad  \boldsymbol{{f}} \in C[0,T].
	\]
The existence theory for systems of Volterra integral equations of the second kind (see~\cite[Ch.\ 2, Theorem 4.2]{gripenberg1990volterra}) yields a unique $\bfmu \in L^\infty(0,T)$. Using $\bfxi_{tt}=\bfmu$ supplemented by the initial data, we conclude that there exists a unique $\bfxi \in W^{2, \infty}(0,T)$ and thus $\uh \in W^{2, \infty}(0,T; V_h)$. Owing to the higher regularity of the coefficient $\alphah$ and source term $f$ in time, an additional bootstrap argument shows that $\bfxi_{tt} \in W^{1,1}(0,T)$
and thus 
 $\uh \in W^{3,1}(0,T; V_h)$.
\end{appendices}

%% file: Main.bbl
\begin{thebibliography}{10}

\bibitem{abdulle2020effective}
{\sc A.~Abdulle and T.~Pouchon}, {\em Effective models and numerical
  homogenization for wave propagation in heterogeneous media on arbitrary
  timescales}, Foundations of Computational Mathematics, 20 (2020),
  pp.~1505--1547.

\bibitem{agmon2010lectures}
{\sc S.~Agmon}, {\em Lectures on elliptic boundary value problems}, vol.~369,
  American Mathematical Society, 2010.

\bibitem{alnaes2015fenics}
{\sc M.~Aln{\ae}s, J.~Blechta, J.~Hake, A.~Johansson, B.~Kehlet, A.~Logg,
  C.~Richardson, J.~Ring, M.~E. Rognes, and G.~N. Wells}, {\em The {FE}ni{CS}
  project version 1.5}, Archive of numerical software, 3 (2015).

\bibitem{alnaes2014unified}
{\sc M.~S. Aln{\ae}s, A.~Logg, K.~B. {\O}lgaard, M.~E. Rognes, and G.~N.
  Wells}, {\em Unified form language: {A} domain-specific language for weak
  formulations of partial differential equations}, ACM Transactions on
  Mathematical Software (TOMS), 40 (2014), pp.~1--37.

\bibitem{baker2022numerical}
{\sc K.~Baker, L.~Banjai, and M.~Ptashnyk}, {\em Numerical analysis of a
  time-stepping method for the westervelt equation with time-fractional
  damping}, Mathematics of Computation,  (2024).

\bibitem{brenner2008mathematical}
{\sc S.~C. Brenner, L.~R. Scott, and L.~R. Scott}, {\em The mathematical theory
  of finite element methods}, vol.~3, Springer, 2008.

\bibitem{degond2017asymptotic}
{\sc P.~Degond and F.~Deluzet}, {\em Asymptotic-preserving methods and
  multiscale models for plasma physics}, Journal of Computational Physics, 336
  (2017), pp.~429--457.

\bibitem{dorfler2016local}
{\sc W.~D{\"o}rfler, H.~Gerner, and R.~Schnaubelt}, {\em Local well-posedness
  of a quasilinear wave equation}, Applicable Analysis, 95 (2016),
  pp.~2110--2123.

\bibitem{dorich1173strong}
{\sc B.~D{\"o}rich}, {\em Strong norm error bounds for quasilinear wave
  equations under weak {CFL}-type conditions}, Foundations of Computational
  Mathematics,  (2024), pp.~1--48.

\bibitem{d2024homogenization}
{\sc L.~D’elia, M.~Eleuteri, E.~Zappale, et~al.}, {\em Homogenization of
  supremal functionals in the vectorial case (via {L}p-approximation)},
  Analysis and Applications,  (2024), pp.~1--48.

\bibitem{Eggermont1987}
{\sc P.~P.~B. Eggermont}, {\em On {G}alerkin methods for {A}bel-type integral
  equations}, SIAM Journal on Numerical Analysis, 25 (1987), pp.~1093--1117.

\bibitem{feireisl2018asymptotic}
{\sc E.~Feireisl, M.~Luk{\'a}cov{\'a}-Medvidov{\'a}, S.~Necasov{\'a},
  A.~Novotn{\'y}, and B.~She}, {\em Asymptotic preserving error estimates for
  numerical solutions of compressible {N}avier--{S}tokes equations in the low
  {M}ach number regime}, Multiscale Modeling \& Simulation, 16 (2018),
  pp.~150--183.

\bibitem{filbet2013analysis}
{\sc F.~Filbet and A.~Rambaud}, {\em Analysis of an asymptotic preserving
  scheme for relaxation systems}, ESAIM: Mathematical Modelling and Numerical
  Analysis, 47 (2013), pp.~609--633.

\bibitem{gorenflo2020mittag}
{\sc R.~Gorenflo, A.~A. Kilbas, F.~Mainardi, and S.~V. Rogosin}, {\em
  Mittag-{L}effler functions, related topics and applications}, Springer, 2020.

\bibitem{gripenberg1990volterra}
{\sc G.~Gripenberg, S.-O. Londen, and O.~Staffans}, {\em Volterra integral and
  functional equations}, no.~34, Cambridge University Press, 1990.

\bibitem{hochbruck2022error}
{\sc M.~Hochbruck and B.~Maier}, {\em Error analysis for space discretizations
  of quasilinear wave-type equations}, IMA Journal of Numerical Analysis, 42
  (2022), pp.~1963--1990.

\bibitem{jager2023approximation}
{\sc W.~J{\"a}ger, A.~Tambue, and J.~L. Woukeng}, {\em Approximation of
  homogenized coefficients in deterministic homogenization and convergence
  rates in the asymptotic almost periodic setting}, Analysis and Applications,
  21 (2023).

\bibitem{jin2019numerical}
{\sc B.~Jin, R.~Lazarov, and Z.~Zhou}, {\em Numerical methods for
  time-fractional evolution equations with nonsmooth data: a concise overview},
  Computer Methods in Applied Mechanics and Engineering, 346 (2019),
  pp.~332--358.

\bibitem{jin2022asymptotic}
{\sc S.~Jin}, {\em Asymptotic-preserving schemes for multiscale physical
  problems}, Acta Numerica, 31 (2022), pp.~415--489.

\bibitem{kaltenbacher2009global}
{\sc B.~Kaltenbacher and I.~Lasiecka}, {\em Global existence and exponential
  decay rates for the {W}estervelt equation}, Discrete \& Continuous Dynamical
  Systems-S, 2 (2009), p.~503.

\bibitem{kaltenbacher2023limiting}
{\sc B.~Kaltenbacher, M.~Meliani, and V.~Nikoli{\'c}}, {\em The {K}uznetsov and
  {B}lackstock equations of nonlinear acoustics with nonlocal-in-time
  dissipation}, Applied Mathematics \& Optimization, 89 (2024), pp.~1--37.

\bibitem{kaltenbacher2022limiting}
\leavevmode\vrule height 2pt depth -1.6pt width 23pt, {\em Limiting behavior of
  quasilinear wave equations with fractional-type dissipation}, Advanced
  Nonlinear Studies, 24 (2024), pp.~748--774.

\bibitem{kaltenbacher2022parabolic}
{\sc B.~Kaltenbacher and V.~Nikoli\'c}, {\em Parabolic approximation of
  quasilinear wave equations with applications in nonlinear acoustics}, SIAM
  Journal on Mathematical Analysis, 54 (2022), pp.~1593--1622.

\bibitem{kaltenbacher2022inverse}
{\sc B.~Kaltenbacher and W.~Rundell}, {\em On an inverse problem of nonlinear
  imaging with fractional damping}, Mathematics of Computation, 91 (2022),
  pp.~245--276.

\bibitem{kaltenbacher2022fractional}
{\sc B.~Kaltenbacher and A.~Schlintl}, {\em Fractional time stepping and
  adjoint based gradient computation in an inverse problem for a fractionally
  damped wave equation}, Journal of Computational Physics, 449 (2022),
  p.~110789.

\bibitem{kaltenbacher2007numerical}
{\sc M.~Kaltenbacher}, {\em Numerical simulation of mechatronic sensors and
  actuators}, vol.~3, Springer, 2014.

\bibitem{kennedy2005high}
{\sc J.~E. Kennedy}, {\em High-intensity focused ultrasound in the treatment of
  solid tumours}, Nature reviews cancer, 5 (2005), pp.~321--327.

\bibitem{kubica2020time}
{\sc A.~Kubica, K.~Ryszewska, and M.~Yamamoto}, {\em Time-fractional
  Differential Equations: A Theoretical Introduction}, Springer, 2020.

\bibitem{lighthill1956viscosity}
{\sc M.~J. Lighthill}, {\em Viscosity effects in sound waves of finite
  amplitude}, Surveys in mechanics, 250351 (1956).

\bibitem{lubich1986discretized}
{\sc C.~Lubich}, {\em Discretized fractional calculus}, SIAM Journal on
  Mathematical Analysis, 17 (1986), pp.~704--719.

\bibitem{lubich1988convolution}
{\sc C.~Lubich}, {\em Convolution quadrature and discretized operational
  calculus. {I}}, Numerische Mathematik, 52 (1988), pp.~129--145.

\bibitem{makridakis1993finite}
{\sc C.~G. Makridakis}, {\em Finite element approximations of nonlinear elastic
  waves}, Mathematics of computation, 61 (1993), pp.~569--594.

\bibitem{Meliani2022}
{\sc M.~Meliani and V.~Nikoli{\'c}}, {\em Mixed approximation of nonlinear
  acoustic equations: {W}ell-posedness and a priori error analysis}, Applied
  Numerical Mathematics, 198 (2024), pp.~94--111.

\bibitem{meyer2011optimal}
{\sc S.~Meyer and M.~Wilke}, {\em Optimal regularity and long-time behavior of
  solutions for the {W}estervelt equation}, Applied Mathematics \&
  Optimization, 64 (2011), pp.~257--271.

\bibitem{muhr2017isogeometric}
{\sc M.~Muhr, V.~Nikoli{\'c}, B.~Wohlmuth, and L.~Wunderlich}, {\em
  Isogeometric shape optimization for nonlinear ultrasound focusing}, Evolution
  Equations \& Control Theory, 8 (2019), p.~163.

\bibitem{nikolic2019priori}
{\sc V.~Nikoli\'c and B.~Wohlmuth}, {\em A priori error estimates for the
  finite element approximation of {W}estervelt's quasi-linear acoustic wave
  equation}, SIAM Journal on Numerical Analysis, 57 (2019), pp.~1897--1918.

\bibitem{ortner2007discontinuous}
{\sc C.~Ortner and E.~S{\"u}li}, {\em Discontinuous {G}alerkin finite element
  approximation of nonlinear second-order elliptic and hyperbolic systems},
  SIAM Journal on Numerical Analysis, 45 (2007), pp.~1370--1397.

\bibitem{owhadi2008numerical}
{\sc H.~Owhadi and L.~Zhang}, {\em Numerical homogenization of the acoustic
  wave equations with a continuum of scales}, Computer Methods in Applied
  Mechanics and Engineering, 198 (2008), pp.~397--406.

\bibitem{podlubny1998fractional}
{\sc I.~Podlubny}, {\em Fractional differential equations: an introduction to
  fractional derivatives, fractional differential equations, to methods of
  their solution and some of their applications}, Elsevier, 1998.

\bibitem{schenke2022explicit}
{\sc S.~Schenke, F.~Sewerin, B.~van Wachem, and F.~Denner}, {\em Explicit
  predictor--corrector method for nonlinear acoustic waves excited by a moving
  wave emitting boundary}, Journal of Sound and Vibration, 527 (2022),
  p.~116814.

\bibitem{szabo2004diagnostic}
{\sc T.~L. Szabo}, {\em Diagnostic ultrasound imaging: inside out}, Academic
  press, 2004.

\bibitem{thomee2007galerkin}
{\sc V.~Thom{\'e}e}, {\em Galerkin finite element methods for parabolic
  problems}, vol.~25, Springer Science \& Business Media, 2007.

\bibitem{westervelt1963parametric}
{\sc P.~J. Westervelt}, {\em Parametric acoustic array}, The Journal of the
  Acoustical Society of America, 35 (1963), pp.~535--537.

\bibitem{zeidler1993nonlinear}
{\sc E.~Zeidler}, {\em Nonlinear Functional Analysis and Its Applications {I}:
  {F}ixed-point Theorems/{T}ransl. by {P}eter {R}. {W}adsack}, Springer-Verlag,
  1993.

\end{thebibliography}
